\documentclass[11pt]{amsart}

\newif\ifHideFoot
\HideFoottrue % This line can be left alone.
\HideFootfalse % Comment this line to hide footnotes

\usepackage{fullpage}
\usepackage{amssymb,amsmath,amsthm,amscd,mathrsfs,graphicx, color,mathtools}
\usepackage[cmtip,all,matrix,arrow,tips,curve]{xy}
\usepackage[colorlinks=true]{hyperref}
\usepackage{cleveref}
\usepackage{tikz}
\usetikzlibrary{fadings,patterns,shadows.blur,shapes,calc,matrix,arrows,shapes,decorations.pathmorphing,decorations.markings,decorations.pathreplacing}
\usepackage{pdfpages}
\usepackage{mathpazo}

\newif\iffullversion
\fullversiontrue % Show details
\numberwithin{equation}{section}
\newtheorem{teo}{Theorem}[section]
\newtheorem{pro}[teo]{Proposition}
\newtheorem{lem}[teo]{Lemma}
\newtheorem{cor}[teo]{Corollary}

\theoremstyle{definition}

\theoremstyle{remark}
\newtheorem{rem}[teo]{Remark}

\newcommand{\calB}{\mathcal{B}}

\newcommand{\calH}{\mathcal{H}}
\newcommand{\calO}{\mathcal{O}}

\newcommand{\calM}{\mathcal{M}}
\newcommand{\calN}{\mathcal{N}}

\newcommand{\Eff}{\operatorname{Eff}}
\newcommand{\Nef}{\operatorname{Nef}}

\newcommand{\GL}{\operatorname{GL}}
\newcommand{\NS}{\operatorname{NS}}

\newcommand{\SL}{\operatorname{SL}}

\newcommand{\Pic}{\operatorname{Pic}}

\newcommand{\ZZ}{\mathbb{Z}}

\newcommand{\QQ}{\mathbb{Q}}
\newcommand{\RR}{\mathbb{R}}
\newcommand{\CC}{\mathbb{C}}

\newcommand{\PP}{\mathbb{P}}

\newcommand{\gquot}{/\!\!/}

\newcommand{\BBG}{{(\calB_4/\Gamma)^*}}

\newcommand{\BBGell}{{(\calB_4/\Gamma_\ell)^*}}
\newcommand{\oBG}{{\overline{\calB_4/\Gamma}}}
\newcommand{\oBGm}{{\overline{\calB_4/\Gamma_m}}}
\newcommand{\oBGell}{{\overline{\calB_4/\Gamma_\ell}}}

\newcommand{\GIT}[1][\calM]{{{#1}^{\operatorname{GIT}}}}
\newcommand{\MK}[1][\calM]{{{#1}^{\operatorname{K}}}}

\newcommand{\tcH}{\widetilde\calH}
\newcommand{\tMq}{\widetilde\calM_{0,7}}
\newcommand{\TAout}{T_{A_1,\ell}^{\rm{out}}}
\newcommand{\TAin}{T_{A_1,\ell}^{\rm{in}}}
\newcommand{\TRout}{T_{R,\ell}^{\rm{out}}}
\newcommand{\TRin}{T_{R,\ell}^{\rm{in}}}

\ifHideFoot

\newcommand{\Yano}[1]{}
\newcommand{\Klaus}[1]{}
\newcommand{\Sam}[1]{}

\else

\newcommand{\marg}[1]{\normalsize{{\color{red}\footnote{#1}}}{\marginpar[\vskip -.3cm {\color{red}\hfill\tiny{\thefootnote}\normalsize$\implies$}]{\vskip -.3cm{ \color{red}$\impliedby$\tiny\thefootnote}}}}
\newcommand{\Yano}[1]{\marg{\color{red}(Yano): #1}}
\newcommand{\Klaus}[1]{\marg{\color{magenta}(Klaus): #1}}
\newcommand{\Sam}[1]{\marg{\color{blue}(Sam): #1}}

\fi

\begin{document}

%%%%%%%%%%%%%%%%%%%%%%%%%%%%%%%%%%%%%%%%%%%%%%% AUTHOR, TITLE, ABSTRACT INFO %%%%%%%%%%%%%%%%%%%%%%%%%%%%%%%%%%%%%%%%%%%%%%%%%%%%%%%%%%%%%%%%%%%%%

\title{The birational geometry of moduli of cubic surfaces and cubic surfaces with a line}

\author[S. Casalaina-Martin]{Sebastian Casalaina-Martin}
\address{Department of Mathematics, University of Colorado, Boulder, CO 80309, USA}
\email{casa@math.colorado.edu}

\author[S. Grushevsky]{Samuel Grushevsky}
\address{Department of Mathematics and Simons Center for Geometry and Physics, Stony Brook University,  Stony Brook, NY 11794-3651, USA}
\email{sam@math.stonybrook.edu}

\author[K. Hulek]{Klaus Hulek}
\address{Institut f\"ur Algebraische Geometrie, Leibniz Universit\"at Hannover, 30060 Hannover, Germany}
\email{hulek@math.uni-hannover.de}

\thanks{Research of the first author is supported in part by a grant from the Simons Foundation (581058).
Research of the second author is supported in part by NSF grant DMS-21-01631. Research of the third author is supported in part by DFG grant Hu-337/7-2.
}

\begin{abstract}
We determine the cones of effective and nef divisors on the toroidal compactification of the ball quotient model of the moduli space of complex cubic surfaces with a chosen line. From this we also compute the corresponding cones for the moduli space of unmarked cubics surfaces.
\end{abstract}

%%%%%%%%%%%%%%%%%%%%%%%%%%%%%%%%%%%%%%%%%%%%%%%%%%%%%%%% MAKE TITLE %%%%%%%%%%%%%%%%%%%%%%%%%%%%%%%%%%%%%

\date{}

\maketitle

%%%%%%%%%%%%%%%%%%%%%%%%%%%%%%%%%%%%%%%%%%%%%%%%%%%%%%%% SUB-FILES %%%%%%%%%%%%%%%%%%%%%%%%%%%%%%%%%%%%%
\section{Introduction}\label{Sec:intronew}
The moduli space $\calM$ of cubic surfaces is a very classical and well-studied space. By invariant theory, see \cite{DvGK} for a modern account, it is isomorphic to an open set in the weighted projective space~$\PP(1,2,3,4,5)$, and by the work of Allcock, Carlson and Toledo \cite{ACTsurf}, it also has a ball quotient model. These identifications extend to the boundary, more precisely, there are isomorphisms
\begin{equation}\label{basicisomorphisms} 
\GIT \cong\PP(1,2,3,4,5) \cong \BBG \,,
\end{equation}
where $\Gamma$ denotes a suitable arithmetic unitary group acting on the $4$-dimensional ball $\calB_4$, and $\BBG$ stands for the Baily-Borel compactification of~$\calB_4/\Gamma$. This isomorphism maps the unique
polystable orbit, which is given by the (unique, up to isomorphism) cubic surface with three~$A_2$ singularities, to the unique cusp~$c$.

GIT quotients or, if they admit ball quotient models, their Baily-Borel compactifications are typically too ``small'' to be suitable for modular interpretations. It is often also the case that the polystable orbits or cusps are badly singular points of compactifications, although in this particular case the cusp~$c\in\BBG$ is in fact a smooth point. For this reason one is naturally led to study ``desingularizations'', which, by construction, have at worst finite quotient singularities. The two obvious candidates for this in our case are the Kirwan blowup $\MK$ of~$\GIT$, and the toroidal compactification $\oBG$ (which in the case of a ball quotient is uniquely determined, since the fan in question is $1$-dimensional, and thus no choices are involved). This gives rise to the diagram
\begin{equation}\label{diag_unmarked}
\xymatrix{
\MK \ar_{\pi}[d]\ar@{-->}^f[r]&\oBG \ar^p[d]\\ \GIT \ar[r]^{\sim}&\BBG
}
\end{equation}
and one is naturally led to the question of whether $f$ is an isomorphism? Possible evidence that this might be the case came from our observation in \cite{cubics} that the Betti numbers $b_i(\MK)=b_i(\oBG),\ i\geq 0$
are the same. However, we showed in \cite{Nonisom} that neither the map~$f$ nor its inverse~$f^{-1}$ extend to a morphism. Moreover, we proved in that paper that~$\MK$ and~$\oBG$ are not isomorphic as abstract varieties. They are, however equivalent in the Grothendieck ring.

Interestingly enough, the situation changes significantly when one goes to the cover of {\em marked} cubic surfaces, by which we mean an identification of the configuration of the 27 lines on a cubic surface with a given fixed abstract configuration, see \cite{naruki}. Then replacing the Kirwan blow-up by the Naruki compactification $\overline \calN$, it was shown in \cite[Theorem 1.4]{GKS} that one obtains an isomorphism
\begin{equation}\label{equ:isomoNarukibal}
\overline{\calN} \cong \oBGm
\end{equation}
where $\Gamma_m$ is the unitary group associated to marked cubic surfaces. The quotient $\Gamma / \Gamma_m \cong W(E_6)$ is the Weyl group of $E_6$, which acts on the configuration of the 27 lines on a 
cubic surface. Clearly, some interesting birational geometry appears here, and the current paper is a contribution to shed further light on this. 

In this paper our main focus will be on cubics with a marked line and we denote $\calM_{\rm{sm},\ell}$ the moduli of smooth cubic surfaces with one marked line. 
There are then forgetful covering maps
\begin{equation}
 \calM_{\rm{sm},\rm{m}} \to \calM_{\rm{sm},\ell} \to \calM_{\rm{sm}} 
\end{equation}
starting with the moduli space of smooth marked cubic surfaces and ending with the moduli spaces of smooth cubic surfaces. 
The covering $\calM_{\rm{sm},\rm{m}} \to \calM_{\rm{sm}}$ is a Galois cover with deck group~$W(E_6)$,
The map $\calM_{\rm{sm},\rm{m}} \to \calM_{\rm{sm},\ell}$ is given by the action of a maximal subgroup $W(D_5) \subset W(E_6)$, which is the stabilizer of as fixed line $\ell$. Note, however, that this is not a normal subgroup. 
Clearly, the map $\calM_{\rm{sm},\ell} \to \calM_{\rm{sm}}$ is an unramified 
$27:1$ cover, but since $W(D_5)$ is not normal in $W(E_6)$, this is not a Galois cover. 
As we shall discuss in \Cref{sec:modulisurfacesline}, these maps extend to (toroidal) compactifications
\begin{equation}\label{equ:towermudispaces}
\overline{\calN}\cong \oBGm \to \overline{\calM}_{\ell} \coloneqq \oBGell \stackrel{\ell}{\to} \overline{\calM}\coloneqq\oBG.
\end{equation} 

The main results of this paper concern the cones of effective and nef divisors of the spaces $\overline{\calM}$ and $ \overline{\calM}_{\ell}$. To state the results, we must introduce some key divisors on these spaces. The discriminant is the divisor of cubic surfaces with nodal singularities; these are stable, and we denote the closure of this locus in $\GIT\cong \BBG$ by $D$, and its strict transform in $\oBG$ by $T_{A_1}$. While $T_{A_1}$ is irreducible, its preimage $T_{A_1,m}$ in $\overline{\calN} \cong \oBGm$ consists of 36 components, permuted by the action of $W(E_6)$. 
The preimage 
\begin{equation}\label{equ:inverseimageellcubicnode}
T_{A_1,\ell} \coloneqq \ell^{-1}(T_{A_1}) = \TAin \cup \TAout \subset \overline\calM_\ell
\end{equation}
of the discriminant $T_{A_1}\subset\overline\calM$ decomposes into two irreducible components, where $\TAin$ is the closure of the locus of cubic surfaces $S$ with a unique node $P$ and a line $\ell$ containing $P$ (so that $P$ is {\em in} $\ell$), and similarly $\TAout$ is the closure of the locus of nodal cubics $S$ with a line $\ell \subset S$ not containing the node $P$, so that $P$ is {\em out} of $\ell$.
Further, we denote the (irreducible) toric boundary of $\oBG$ by $T_{3A_2}$, and denote its preimages in $\oBGell$ and $\oBGm$ by $T_{3A_2,\ell}$ and $T_{3A_2,m}$ respectively. 
We shall see that~$T_{3A_2,\ell}$ is in fact irreducible, while $T_{3A_2,m}$ is known to have 40 irreducible components. Another important class is the class of the Hodge line bundle $\lambda$ on $\BBG$. 
We shall also denote the pullbacks of this class to
$\oBG$, $\oBGell$ and $\oBGm$ by the same letter, as it will always be clear on which space we work. 
The classes $T_{A_1}, T_{3A_2}$ and $\lambda$ are linearly dependent on $\oBG$, namely 
$T_{A_1}=24 \lambda - T_{3A_2}$. Another geometrically interesting divisor is given by the (closure of) the locus of cubic surfaces with an Eckardt point. 
We denote the closures of this locus in~$\oBG$ and~$\MK$ by $T_R$ and $D_R$, respectively. 
These are irreducible divisors. We also note that $ T_R=150 \lambda - 24 T_{3A_2}$. Finally we recall that $K_{\oBG}=5\lambda - \frac{5}{6} T_{A_1} - \frac{1}{2} T_R - T_{3A_2}$. 
For a more comprehensive discussion of these relations we refer to \Cref{pro:pullbacks}.

We can now describe the cones of effective and nef divisors, starting with the unmarked case.

\begin{teo}\label{teo:effandnef}
The cones of nef and effective divisors on the toroidal compactification $\overline\calM=\oBG$ are given by
\begin{equation}\label{equ:eff}
\Eff(\overline{\calM})=\RR_{\ge 0}T_{A_1}+\RR_{\ge 0}T_{3A_2}
\end{equation}
and 
\begin{equation}\label{equ:nef}
\Nef(\overline{\calM})= \RR_{\ge 0}(T_{A_1} + 2T_{3A_2})+\RR_{\ge 0}(T_{A_1} + 6T_{3A_2}).
\end{equation}
\end{teo} 

The proof of \Cref{teo:effandnef} is given at the end of \S\ref{sec:effecitveconeviacurves} for \eqref{equ:eff} and the end of \S \ref{sec:nefcones} for \eqref{equ:nef}.
From  \Cref{teo:effandnef} we can immediately obtain (see \Cref{pro:pullbacks} and \eqref{equ:canonicalbundles}) 
\begin{cor}\label{cor:M}
The discriminant $T_{A_1}$ and the anti-canonical divisor $-K_{\overline{\calM}}$ are effective but not nef. The toric boundary divisor $T_{3A_2}$ is an extremal effective divisor, and the Hodge line bundle $\lambda$ is an extremal nef divisor. The divisor $T_R$ of cubic surfaces with an Eckardt point is both effective and nef.
\end{cor}

In our approach we will deduce this from the corresponding result on the moduli space of cubic surfaces with a line, which is a more refined result, interesting in its own right. Note that $\oBGm \to \oBGell$
is given by a non-normal subgroup and that $\oBGell \to \oBG$ is not Galois. Hence there is no easy way to obtain results on $\oBGell$ from the other two spaces, and vice versa. 

\begin{teo}\label{teo:effnefline}
The following holds:
\begin{itemize}
\item[\rm{(i)}] The Picard group of the moduli space $\overline{\calM}_{\ell}$ of cubic surfaces has rank $3$. The group $\Pic_\QQ(\overline\calM_\ell)$ is generated by $\TAin, \TAout, T_{3A_2,\ell}$.
\item[\rm{(ii)}] The cone of effective divisors is
\begin{equation}\label{equ:effline}
\Eff(\overline{\calM}_\ell)= \RR_{\geq 0} \TAin + \RR_{\geq 0} \TAout + \RR_{\geq 0} T_{3A_2,\ell}\,.
\end{equation}
\item[\rm{(iii)}] The cone of nef divisors is
\begin{equation}\label{erqu:nefline} 
\Nef(\overline{\calM}_\ell) = \{a\TAin+bT_{3A_2,\ell}+c\TAout \mid b\ge a, \, b/2\ge c\ge b/6, \, 2a+2c\ge b \}\,.
\end{equation}
\end{itemize}
\end{teo}

The proof of \Cref{teo:effnefline} is given at the end of \S \ref{sec:nefcones}. 
In a recently posted paper N.~Schock \cite{schock} determined the cones of $W(E_6)$-invariant effective and nef divisors of the Naruki compactification $\overline{\calN}$, which are given by
\begin{teo}[Schock]\label{teo:Minvariant}
The cones of $W(E_6)$-equivariant effective divisors and nef divisors on the Naruki and toroidal compactification $\calN \cong \oBGm$ are 
\begin{equation}\label{equ:effm}
\Eff(\oBGm)=\RR_{\ge 0}T_{A_1,m}+\RR_{\ge 0}T_{3A_2,m}
\end{equation}
and 
\begin{equation}\label{equ:nefm}
\Nef(\oBGm)= \RR_{\ge 0}(T_{A_1,m} + T_{3A_2,m})+\RR_{\ge 0}(T_{A_1,m} + 3T_{3A_2,m}).
\end{equation}
\end{teo}
Theorems \ref{teo:effandnef} and \ref{teo:Minvariant} are equivalent to each other, once one notes that $\overline\calN \cong \oBGm$ and that the map to the unmarked space is $W(E_6)$-equivariant. The difference in the slopes of the cones comes from the fact that the quotient map $\oBGm\to\oBG$ is ramified to order 2 over the discriminant, but unramified over the toric boundary. 

In Sections \Cref{sec:intersectionmKirwan} we shall complete the picture by computing the intersection theory on the Kirwan blow-up $\MK$. We had already determined the intersection theory of divisors, and hence the ring structure on cohomology for $\oBG$ in \cite{Nonisom}. In \Cref{teo:intersectiontheory} we shall do the same for the Kirwan blow-up $\MK$, which immediately tells us that the top self-intersection numbers of the canonical divisor on $\MK$ and on $\oBG$ are different. We argued in \cite{Nonisom} that the denominators of these intersection numbers are different, whereas we 
now determine the actual values of these intersections --- and thus reconfirm that $\MK$ and $\oBG$ are not $K$-equivalent.

 \subsection*{Acknowledgements}
We thank Johannes Schmitt for help with the package \texttt{admcycles} and Mathieu Dutour for computing the volume of a polytope for us. We are further grateful to Radu Laza for many discussions about cubic surfaces and their moduli.
 
\section{The set-up, and the intersection theory}\label{sec:intersection}
We already recalled that $\GIT \cong \BBG \cong\PP(1,2,3,4,5)$, and introduced the Kirwan blow-up $\pi:\MK\to \GIT$ and the toroidal compactification $p:\oBG\to\BBG$. 
Both $\pi$ and $p$ are blowups supported at the cusp $c$ corresponding to the cubic with $3A_2$ singularities, given by the equation $F_{3A_2}(x)=x_0x_1x_2+x_3^3=0$. 
We also recall that $\MK$ and $\oBG$ are smooth (as DM stacks, i.e.~up to finite quotient singularities). 
One of the starting points for our work on the moduli space of cubic surfaces is the observation from
\cite[App.~C]{cubics} that $\MK$ and $\oBG$ have the same cohomology (with $\QQ$ coefficients) --- it is only non-zero in even degrees, and the dimensions are
$$
\begin{array}{r|rrrrr}
&h^0&h^2&h^4&h^6&h^8\\
\hline
\MK&1&2&2&2&1\\
\oBG&1&2&2&2&1
\end{array}
$$
These spaces are smooth up to finite quotient singularities and hence satisfy Poincar\'e duality. The weighted projective space $\PP(1,2,3,4,5)$ comes with a natural ($\ZZ$-) line bundle $\calO_{\PP(1,2,3,4,5)}(1)$ defined as the positive generator of the 
Picard group of $\PP(1,2,3,4,5)$. As a Baily-Borel compactification, this space also carries another natural ($\QQ)$-line bundle, namely the Hodge line bundle $\lambda$ corresponding to modular forms of weight 1. We recall from \cite[(5.2)]{Nonisom} that 
\begin{equation}\label{equ:Hodge}
6 \lambda = \calO_{\PP(1,2,3,4,5)}(1).
\end{equation}
The discriminant divisor $D$ and the divisor of Eckardt cubics $R$ must be multiples of $\lambda$. In our previous paper \cite[(3.9)]{Nonisom} we showed that 
\begin{equation}\label{equ:discrimninant}
D = \calO_{\PP(1,2,3,4,5)}(4)
\end{equation}
and 
\begin{equation}\label{equ:classR}
R =\calO_{\PP(1,2,3,4,5)}(25).
\end{equation}
This follows from classical invariant theory going back to Clebsch and Salomon, see \cite[Section 9.4.5]{dolclassical}. 

Recalling the notation from \eqref{diag_unmarked}, the exceptional divisors of~$\pi$ and~$p$ are called $D_{3A_2}$ and $T_{3A_2}$, respectively. Both of these are irreducible. The strict transforms of the discriminant under $\pi$ and $p$ are called $D_{A_1}$ and $T_{A_1}$, respectively. Here we use the notation $T$ for ``toroidal'' to always be clear where we work, but will still think of $T_{A_1}$ as a Heegner divisor. Accordingly, we denote the strict transforms of $R$ in $\MK$ and $\oBG$ by $D_R$ and $T_R$. Finally $\lambda$ denotes the Hodge line bundle on $\BBG$ as well as its pullback to the various blow-ups and covers. 

The relationship of the divisors described above is given by
\begin{pro}\label{pro:pullbacks}
The following holds
\begin{itemize}
\item[({\rm i})] $D_{A_1}=\pi^*D - 6D_{3A_2} = 24 \lambda - 6D_{3A_2}, \quad D_R= \pi^*R - 30D_{3A_2} = 150 \lambda - 30D_{3A_2} $,
\item[({\rm ii})] $T_{A_1}=p^*D - 6T_{3A_2} = 24 \lambda - 6T_{3A_2}, \quad T_R= p^*R - 24T_{3A_2}= 150 \lambda - 24T_{3A_2} $.
\end{itemize}
\end{pro}
The formulae for $\oBG$ were obtained in \cite[Remark 5.9]{Nonisom}, whereas the results for $\MK$ are new, and we will deduce them in \Cref{sec:intersectionmKirwan}.

We thus have 
$$
 H^2(\MK,\QQ)= \QQ D_{A_1}\oplus \QQ D_{3A_2} = \QQ \lambda \oplus \QQ D_{3A_2}\,;
 $$
$$
 H^2(\oBG,\QQ)=\QQ T_{A_1} \oplus \QQ T_{3A_2}= \QQ \lambda \oplus \QQ T_{3A_2}\,.
$$

These divisors generate the cohomology rings of $\MK$ and $\oBG$, and can be used to completely describe the intersection theory of these spaces. 
\begin{teo} \label{teo:intersectiontheory}
The intersection theory of $\MK$ and $\oBG$ is given by
 \begin{itemize}
\item[({\rm i})] on~$\MK$: $\lambda^4= \frac{1}{5!6^4}= \frac{1}{155520}$, \ $\lambda^aD_{3A_2}^{4-a}=0\hbox{ unless }a(4-a)=0$, 
\ $D_{3A_2}^4=- \frac{1}{9 \cdot 56}$.
\item[({\rm ii})] on~$\oBG$: $\lambda^4= \frac{1}{5!6^4}= \frac{1}{155520}$, \ $\lambda^aT_{3A_2}^{4-a}=0\hbox{ unless }a(4-a)=0$, 
\ $T_{3A_2}^4= - \frac{1}{6^3}$.
\end{itemize}
\end{teo}
Here, the intersection number $\lambda^4=\calO^4_{\PP(1,2,3,4,5)}(1/6)$ follows from standard results on weighted projective spaces. The top intersection number $T_{3A_2}^4$ was computed in \cite[Lemma 5.1]{Nonisom}, while the computation of $D_{3A_2}^4$ is new, and we perform it in 
\Cref{sec:intersectionmKirwan}. The vanishing of the remaining intersection numbers follows simply from the fact that $\lambda$ is a pullback from $\GIT$. 
A crucial role is, naturally, also played by the canonical bundles (which like all line bundles on our DM stacks have to be considered as $\QQ$-line bundles). Here we collect the main expressions for the canonical bundles
\begin{equation}\begin{aligned}\label{equ:canonicalbundles} 
K_{\GIT}&= 5\lambda - \frac{5}{6}D - \frac{1}{2}R = \calO_{\PP(1,2,3,4,5)}(-15) = -90 \lambda\,,\\
K_{\MK} &= \pi^* K_{ \GIT}+20D_{3A_2}=5\lambda - \frac{5}{6} D_{A_1} - \frac{1}{2}D_R +40D_{3A_2}\,,\\
K_{\oBG}&=5\lambda - \frac{5}{6} T_{A_1} - \frac{1}{2} T_R - T_{3A_2}= p^*K_{\BBG}+16T_{3A_2}\,.
\end{aligned}
\end{equation}
The first equation comes from the ball quotient picture, together with the fact that the map $\calB_4 \to \calB_4/\Gamma$ is ramified, with ramification orders 6 and 2 along the discriminant and the Eckardt divisor, respectively.
The expression for $K_{\GIT}$ 
in terms of $\calO_{\PP(1,2,3,4,5)}(1)$ can be derived from the theory of weighted projective spaces, and the last equality in that line  
follows from (\ref{equ:Hodge}).
The first equality for $K_{\MK}$ was shown in \cite[Corollary 6.8]{Nonisom}, and the second can then be deduced from this and Proposition \ref{pro:pullbacks}(\rm{i}). Finally, the first equation for $K_{\oBG}$ is standard
from the theory of toroidal compactifications, and the second follows from this and Proposition \ref{pro:pullbacks}(\rm{ii}). 

Our argument in \cite{Nonisom} to prove that~$\MK$ and~$\oBG$ were not isomorphic, or even $K$-equivalent, was by showing that $K^4_{\MK} \neq K^4_{\oBG}$ by arguing that these rational numbers (in reduced form)   
have different denominators. Knowing the intersection theory of both spaces, we can now complete this by producing the actual top self-intersection numbers.
\begin{pro}\label{pro:intersect}
The top self-intersection numbers of the canonical divisors are given by
$$\ \ K_{\MK}^4 = K_{\GIT}^4 + (20D_{3A_2})^4=\frac{3375}{8} - \frac{20^4}{9\cdot 56}= \frac{3375}{8} - \frac{20000}{63} 
$$
$$
\ne K_{\oBG}^4 = K_{\BBG}^4 + (16T_{3A_2})^4 = \frac{3375}{8} - \frac{16^4}{6^3} = \frac{3375}{8} - \frac{8192}{27}.
$$
\end{pro}
\begin{proof}
This is an immediate consequence of the previous calculations.
\end{proof}
The above computations allow us to give
\begin{proof}[Proof of \Cref{cor:M}]
The claims about $T_{A_1}$ follow from \Cref{teo:effandnef}, using the class computations from \Cref{pro:pullbacks}. Since the birational map~$p:\oBG\to\BBG$ contracts the effective divisor~$T_{3A_2}$ to a point, it follows that $T_{3A_2}$ is an extremal effective divisor. The Hodge line bundle $L$ is the pullback of an ample line bundle under the map $p:\oBG\to\BBG$. It is positive on all curves not contained in the toric boundary, but trivial on curves lying in the boundary. 
The claims about the other two divisors follow from standard calculations using the results already mentioned: 
\begin{equation}\label{eq:TR}\begin{aligned}
8T_R&=150\cdot 8\lambda-24\cdot 8T_{3A_2}=\tfrac{150\cdot 8}{24}(T_{A_1}+6T_{3A_2})-192T_{3A_2}\\ &=50T_{A_1}+(300-192)T_{A_2} =50T_{A_1}+108T_{3A_2}\end{aligned}
\end{equation}
so the "slope" of $T_R$ (ratio of coefficients) is $108/50=2.16$, which is between $2$ and $6$;
$$\begin{aligned}
-4K_{\oBG}&=-4(p^*K_{\BBG}+16 T_{3A_2})=60\calO(1)-64T_{3A_2}=15 (T_{A_1}+6T_{3A_2})-64 T_{3A_2}\\
&=15 T_{A_1}+(90-64)T_{3A_2}=15T_{A_1}+26T_{3A_2}
\end{aligned}
$$
so the slope is $26/15=1.73 <2$.
\end{proof}

\section{The moduli space of cubics with a line}\label{sec:modulisurfacesline}
In this section we start our detailed investigation of the moduli space $\calM_{\ell}$ of cubic surfaces $S$ with a chosen line $\ell \subset S$. 

Recall that we have the following covers of the moduli space of cubics, starting with the moduli space $\calM_{\rm{m}}$ of marked cubic surfaces, ending with the moduli space $\calM$ of cubics, and with the moduli space $\calM_{\ell}$ of cubic surfaces with a line as an intermediate space. Here the index~m stands for marked,~$\ell$ stands for a chosen line, the subscript sm denotes the moduli of smooth cubics only, and $\Delta$ are the suitable discriminant loci. For details we refer to \cite{doran}, \cite{DvGK} and \cite{GKS}. 
\begin{equation}\label{relations}
 \xymatrix{
 \calM_{\rm{sm},\rm{m}} \cong ((\PP^2)^6 \setminus \Delta_{\rm{sp}}) \gquot \SL(3, \CC) \cong (\calB_4 \setminus \Delta_{\rm{m}})/\Gamma_{4,\rm{m}}
 \ar[d]^{16:1}
\\
\calM_{\rm{sm},\rm{m}}/(\ZZ/2\ZZ)^4 \cong {\rm{DM}^0(2^5,1^2)} / S_2 \cong (\calB_4 \setminus \Delta'_{\rm{m}})/\Gamma'_{4,\rm{m}}
\ar[d]^{120:1}
\\ \calM_{\rm{sm},\ell} \cong {\rm{DM}^0(2^5,1^2)} / (S_2 \times S_5) \cong (\calB_4 \setminus \Delta_{\rm{sm},\ell})/\Gamma_{4,\ell}
\ar[d]^{27:1}
\\ \calM_{\rm{sm}} \cong \PP(H^0_{\rm{sm}}(\PP^3, \calO_{\PP^3}(3))) \gquot \SL(4,\CC) \cong (\calB_4 \setminus \Delta_{\rm{sm}})/\Gamma_{4}.
}
\end{equation}
Here the first two maps arise in the following way: we choose a line $\ell$. Then the stabilizer of~$\ell$ within $W(E_6)$ is $W(D_5)\cong (\ZZ/2\ZZ)^4 \rtimes S_5$, and the first two quotients are given by $(\ZZ/2\ZZ)^4$ and $S_5$, respectively. There are two ways of associating $5+2$ points to a marked cubic surface. Details can be found in \cite[Section 3]{DvGK}. Here we recall the basic constructions. Given a marked cubic surface we 
can think of it as a $\PP^2$ blown up in 6 ordered points $p_i,\, i=1, \dots 6$. If we choose any $p_i$, there exists a unique conic $C_i\subset\PP^2$ going through the remaining $5$ points. 
Since $p_i \notin C_i$, there are two tangents lines to $C_i$ going through $p_i$; denote $q_1,q_2\in C_i$ their points of tangency.
Note that this is an unordered pair as we have no way of ordering the two tangents. Projection from $p_i$ then defines a $2:1$ map $C_i\to\PP^1$. The images $\bar p_j$ of the points $p_j, j \neq i$ under this map together with the images $\bar q_i$ of $q_i$ then gives an ordered $5$-tuple of points on this $\PP^1$, plus an unordered pair. 

It is crucial for us that this tower of covers establishes a connection to a Deligne-Mostow variety, namely $\operatorname{DM}(2^5,1^2)$. This means that one considers ordered $7$-tuples of points in $\PP^1$ modulo the group $\SL(2,\CC)$. The weights $(2^5,1^2)$, however, mean that one does not take the standard linearization of the action of $\SL(2,\CC)$, but rather the linearization associated to the action of $\SL(2,\CC)$ on the line bundle $\calO_{\PP^1}(2)^{\boxtimes 5} \boxtimes \calO_{\PP^1}(1)^{\boxtimes 2}$. Using the corresponding notion of stability one defines
\begin{equation*}
 \rm{DM}(2^5,1^2)\coloneqq ((\PP^1)^{7})^{ss} \gquot \SL(2,\CC)\,.
\end{equation*}
On this space the groups $S_5$ and $S_2$ act by permuting the corresponding points. The geometric construction that associates to a marked cubic surface, given by $6$ points in $\PP^2$, a $7$-tuple of points on $\PP^1$ leads to a $16:1$ map
\begin{equation*}
\calM_{\rm{sm},\rm{m}} \cong ((\PP^2)^6 \setminus \Delta_{\rm{sp}}) \gquot \SL(3, \CC) \to {\rm{DM}}^0(2^5,1^2)/S_2
\end{equation*}
In fact there is a group $(\ZZ/2\ZZ)^4$ acting on $\calM_{\rm{sm},\rm{m}}$ which defines this quotient. Here ${\rm{DM}}^0(2^5,1^2)$ 
denotes the open subset of ${\rm{DM}}(2^5,1^2)$ where no two of the seven points coincide. The group $S_5$ acts on both spaces, and the map above is $S_5$-equivariant, thus identifying the moduli space $\calM_{\rm{sm},\ell}$ with the Deligne-Mostow quotient ${\rm{DM}^0(2^5,1^2)}/ (S_2 \times S_5)$. Under this identification, the conic $C_i$ corresponds to the marked line $\ell$. 

Note that the cover $\calM_{\rm{sm},\ell} \to \calM_{\rm{sm}}$ is not Galois. Indeed, the Weyl group $W(E_6)$, which is the Galois group of $\calM_{\rm sm,m}\to\calM_{\rm sm}$, is not a simple group, but its only non-trivial normal subgroup has index $2$. In particular, there is no normal subgroup of $W(E_6)$ which is isomorphic to $W(D_5)$. Also, geometrically, we have chosen a special line to construct the upper part of the diagram of covers.

\begin{rem}
The space ${\rm{DM}}(2^5,1^2)$ is not in Mostow's original list and for this reason does not appear in the tables in \cite{KLW} either. It is, however, in Thurston's amended list \cite{Thurston}. 
\end{rem}

\begin{rem}
There is actually another geometric way of describing these $5+2$ points. For this we choose the line $\ell \subset S$ given by the exceptional divisor which is obtained by blowing up $p_i$. The planes through $\ell$ define a conic bundle $\tilde S \to \PP^1$. This conic bundle has $5$ singular fibers. 
This gives a $5$-tuple of points in $\PP^1$. The singular fibers can be described in terms of lines on
the surface $S$, and the marking on $S$ then allows to define an ordering on the singular fibres and hence on the points $p_j, j \neq i$. 
The conic bundle also has a natural $2$-section branched over the base $\PP^1$ in two points. These give the unordered pair $q_1,q_2$.
\end{rem}

Note that there are a priori different ways of defining a ball quotient structure on the various moduli spaces of cubic surfaces. These are namely (1) the method due to Allcock, Carlson and Toledo \cite{ACTsurf} via the intermediate Jacobians of cubic threefolds, (2) the construction given by Dolgachev, van Geemen and Kondo \cite[Section 3]{DvGK} using configurations of $5+2$ points on a line, establishing a link to Deligne-Mostow varieties, and (3) the approach via moduli spaces of certain lattice-polarized $K3$ surfaces, as explained in \cite[Section 6]{DvGK}. These constructions lead to commensurable ball quotients. The comparison between the Allcock-Carlson-Toledo ball quotient and the Deligne-Mostow approach can be found in \cite[Theorem 3]{doran}, and the relationship of the Dolgachev-van Geemen-Kondo construction using $K3$ surfaces with the work of Allcock-Carlson-Toledo is explained in \cite[Section 6.14]{DvGK}.
 
The diagram \eqref{relations} is only for the moduli spaces of {\em smooth} cubic surfaces. For a discussion about the extension of the maps from the (decorated) moduli spaces of cubic surfaces to the ball quotients we refer to \cite[Section 3]{doran} and \cite[Sections 3 and 9]{DvGK}.
We further note that the constructions of the toroidal compactifications of $\calB_4 / \Gamma_{4,\rm{m}}$ and $\calB_4 / \Gamma_4$ are compatible with the group action of $W(E_6)$.

We now proceed to investigate the geometry of the $27:1$ finite covers
\begin{equation}\label{equ:introcoverell}
\ell: \oBGell\to\oBG;\qquad \ell:\BBGell\to\BBG\,.
\end{equation} 
In the end this will be used to deduce the cones of effective and nef divisors on $\oBG$ from those of $\oBGell$. We start with the discriminant divisor. Since we are working with the strict transform of the discriminant $D\subset\BBG$, we will first work with the Baily-Borel compactification. 

\begin{lem}\label{lem:TA1ell}
The preimage $D_\ell\subset\BBGell$ of the discriminant divisor $D\subset\BBG$, that is the locus of nodal cubics with a chosen line, has two irreducible components, which are distinguished by whether the node of a generic nodal cubic in the component is contained in the chosen line or not. The same holds for the strict transform $T_{A_1,\ell}$ of $D_\ell$ in $\oBGell$.
\end{lem}
\begin{proof}
We start with a generic nodal surface, i.e. with a surface $S$ with a unique node $P_0$. By changing coordinates on $\PP^3$, we move the node of the cubic to be the point $P_0=(1:0:0:0)$. Then the cubic must have an equation
$$
G=x_0Q(x_1,x_2,x_3) + F_3(x_1,x_2,x_3)=0\,.
$$
The space of cubics with a node at $P_0$ is a $16$-dimensional linear subspace in the $20$-dimen\-sio\-nal space $H^0(\PP^3,\calO_{\PP^3}(3))$ of all cubics. We first assume that the line $\ell$ does not contain $P_0$, in which case we can assume that
$$
\ell=\lbrace x_0=x_1=0\rbrace\,.
$$
The condition $\ell \subset S$ then imposes another 4 conditions on the equation of~$S$, namely that the coefficients of $x_2^3, x_2^2x_3, x_2x_3^2, x_3^3$ vanish, giving rise to a $12$-dimensional subspace $V_1$. This is acted upon by the group $G_1 \subset \GL(4,\CC)$ which fixes $P_0$ and the line $\ell$ (as a set). Thus $\dim G_1=9$, and $\PP(V_1) \gquot S_1$ is a $3$-dimensional irreducible variety, where $S_1\coloneqq G_1 \cap \SL(4,\CC)$.

The case where $P \in \ell$ is treated in the same way. Here we can assume that the line is
$$
\ell=\lbrace x_2=x_3=0\rbrace\,.
$$
This leads to a $14$-dimensional linear space $V_2$ and an $11$-dimensional group $G_2$ acting on it, so that the quotient is again irreducible 3-dimensional.

The claim for $T_{A_1,\ell}$ now follows immediately, as we showed each component is irreducible.
\end{proof}
As discussed in \eqref{equ:inverseimageellcubicnode}, we will denote by $\TAin$ the closure of the set of pairs $(S,\ell)$, where $S$ is a cubic surface with one node $P$, which is in the chosen line $\ell$: $P\in\ell$, and denote $\TAout$ the closure of the locus of nodal cubics $S$ with a line $P\notin\ell \subset S$ not containing the node.

We will need to compute the pullback of the divisor $T_{A_1}$ to $\BBGell$, for the determination of the nef cones.
\begin{lem}\label{lem:AB}
The pullback of the divisor $T_{A_1}$ under $\ell$ is
\begin{equation}\label{equ:coeffdiscriminantline}
\ell^*T_{A_1}=2\TAin + \TAout \,.
\end{equation}
\end{lem}
\begin{proof}
Since $\ell: \oBGell\to\oBG$ is a finite map, the pullback of $T_{A_1}$ is given as
$$
\ell^*T_{A_1}=A\TAin + B\TAout 
$$ 
for some integers $A,B >0$. Also, since the map $\ell: \oBGm\to\oBG$ is branched of order $2$ along the discriminant (\cite{naruki}, see also \cite[Thm.~4.6]{Nonisom}), we must have $A,B \in \{1,2\}$. We recall that a general nodal cubic surface contains 21 lines, 6 of which go through the node, while the remaining 15 lie in the smooth locus of the surface, see e.g. \cite[Section 9.2]{dolclassical}. The latter implies that there can be no ramification of order 2 along the divisor $\TAout$, and hence $B=1$. This only leaves the option of $A=2$ to account for the total degree of $\ell$ over $T_{A_1}$ being equal to 27. We remark that this also 
fits with the fact that the 6 lines through the node are counted with multiplicity 2 in the Fano scheme of lines on $S$. 
\end{proof}
\begin{rem}
We denote the preimage of the Eckardt locus by
\begin{equation*}
T_{R,\ell} := \ell^{-1}(T_R).
\end{equation*}
It is clear that this locus has at least two components depending on whether the chosen line goes through the Eckardt point or not. We write this as 
\begin{equation*}\label{equ:componentseckardt}
T_{R,\ell} = \TRin \cup \TRout\,,
\end{equation*}
where $\TRin$ denotes the locus of cubic surfaces with an Eckardt point and the chosen line containing this point, and $\TRout$ means that the chosen line does not contain the Eckardt point. We shall see in the proof of \Cref{pro:comparingclasses} that $\TRin$ is irreducible. Every Eckardt cubic admits an involution fixing a given Eckardt point. For a generic Eckardt cubic, both the Eckardt point and the involution are unique, and there are no further automorphisms. This involution fixes each line through the Eckardt points as a set, but not pointwise, whereas it permutes the 24 lines not going through the Eckardt point pairwise. This can be proved analogously to \cite[Lemma 1.16]{CMMZ}, which deals with Eckardt cubic threefolds. This implies that the map
$\calM_{\rm{sm},\ell}\to \calM_{\rm{sm}}$ sending $(S,\ell)$ to $S$ is not ramified over the component parameterizing pairs $(S,\ell)$ where $\ell$ is a line through the Eckardt point, since $(S,\ell)$ and $S$ both have automorphism group of order $2$. But it is ramified over the locus parameterizing lines $(S,\ell)$ with $\ell$ not through the Eckardt point, because $(S,\ell)$ has trivial automorphism group, but $S$ has automorphism group of order $2$. 
One can also see this ramification via a point count: if one has $(S,\ell)$ with $\ell$ not through the Eckardt point, then the Eckardt involution gives that $(S,\ell)\cong (S,\ell')$, where $\ell'$ is the image of $\ell$ under the Eckardt involution, so that the preimage of a generic point in $T_R$ with $T^{\rm{out}}_{R,\ell}$ consists of $12$ points (rather than $24$).  
We thus obtain that
\begin{equation}\label{equ:pullbackeckardt}
\ell^*(T_{R}) = \TRin + 2\TRout\,. 
\end{equation}
Unlike the corresponding result \Cref{lem:AB} for the pullback of the discriminant divisor, we shall not need this in our proofs. 
\end{rem}

\section{The effective cones}\label{sec:effecitveconeviacurves}
The various moduli spaces of cubic surfaces are closely related to moduli spaces of points on~$\PP^1$. This will be crucial for our calculations, and thus we recall the relevant facts here. 
One of the main results of the paper \cite{GKS} is their Theorem 1.1, which identifies the toroidal compactification of the Deligne-Mostow variety $\rm{DM}(2^5,1^2)$ with the 
compactified Hassett space $\overline\calM_{(1/6+\varepsilon)^2,(1/3+\varepsilon)^5}$ of weighted points on a projective line. For further use we denote this by 
\begin{equation}\label{equ:defninitionH}
\calH\coloneqq \overline\calM_{(1/6+\varepsilon)^2,(1/3+\varepsilon)^5}\,,
\end{equation}
where $\calH$ stands for Hassett's space. We recall that there is a divisorial contraction $\overline\calM_{0,7}\to\calH$. Indeed, Hassett spaces for different weights admit a morphism, if each weight in one is greater or equal than the corresponding weight in the other, and the Deligne-Mumford compactification $\overline\calM_{0,7}=\overline\calM_{(1+\varepsilon)^7}$ 
is itself a Hassett space. Further recall that the Hassett space requires marked points to stay away from the nodes of a stable curve, but allows marked points to come together at a smooth point of a nodal curve as long as the total weight of any merged collection of points is at most 1; the stability condition (we are only interested in genus 0, so all curves are of compact type, and every irreducible component is rational) is to say that on each rational component the total weight of special points is greater than two, where the nodes count with weight 1, and the marked points with the weight prescribed.

The Picard group of $\overline\calM_{0,n}$ was computed by Keel \cite{keel}: it is generated by the classes of the boundary divisors, and Keel determined all the linear relations they satisfy. Using this work, Rulla \cite{rulla} studied the Picard group and cones of effective divisors of quotients $\overline\calM_{0,n}/S_m$, where $S_m,\, m \leq n$ acts on the set of $n$ ordered points. In particular, Rulla showed that the rank of the Picard group $\dim_\QQ\Pic_\QQ(\overline\calM_{0,7}/S_5)=7$, that there is one relation among 
the 8 irreducible components of the boundary of this space, and that the cone of effective divisors $\Eff(\overline\calM_{0,7}/S_5)$ is generated by the irreducible components of the boundary. Following Rulla, we think of points labeled 1 and 2 as 
the ones not permuted by $S_5$, and denote by $D_{12}^i\subset\overline\calM_{0,7}/S_5$ the boundary divisor whose generic point has two irreducible components, first component containing points 1 and 2, and $i-2$ other unlabeled points. 
Similarly, denote $D_{1}^i$ the divisor the generic point of which has two irreducible components, one containing point 1, and $i-1$ other unlabeled points. 
Then the only relation in $\Pic_\QQ(\overline\calM_{0,7}/S_5)$ is given in \cite[Cor.~3.2]{rulla} as
\begin{equation}\label{rulla}
\sum_{i=2}^5 (7-i)(7-i-1)D_{12}^i=\sum_{i=2}^5 (i-1)(7-i-1)D_1^i.
\end{equation}
The coefficients of this relation are crucial for us, and one has to be very careful as to whether such a relation is to be read as a relation between reduced divisors, or in a {\em stacky} sense, where each divisor is weighted by $1/k$ where $k$ is the order of the stabilizer group of a generic point. In fact, Rulla's relation refers to the reduced divisors, whereas we require the stacky version of this relation. 

To be safe, we (re)derive this from Keel's relation, making sure that we write all relations as relations among stacky divisors. In particular, we recall from \cite[Equation (1.5)]{schwarz} that on a quotient $\overline\calM_{0,n}/G$, for $G\subset S_n$, the only boundary divisors whose generic point has a non-trivial automorphism are those parameterizing two-component stable rational curves, where one component contains only two marked points, such that the transposition of these two points is contained in $G$. This amounts to saying that the only divisor on $\overline\calM_{0,7}/S_5$ with a non-trivial stabilizer of a generic point is $D_{12}^5$, while on the quotient $\overline\calM_{0,7}/S_2\times S_5$ this also holds for the image of $D^2_{12}$ (which we will later denote by $D_{\circ\circ}^2$). We now deduce the relations among the (stacky) divisors on these quotients by pushing forward Keel's relation from $\overline\calM_{0,7}$.

Recall that Keel's relation among divisors on $\overline\calM_{0,7}$ is given by
\begin{equation}\label{eq:keel}
 \sum_{1,2\in T\not\ni 3,4}D_T=\sum_{1,3\in T\not\ni 2,4}D_T\,,
\end{equation}
where the sum on each side is taken over all subsets $T\subset \{1,\dots,7\}$. Furthermore, the stabilizer of a generic point of any divisor on $\overline\calM_{0,7}$ is trivial, so 
there are no stacky issues here. We now pushforward this relation under the quotient map $P:\overline\calM_{0,7}\to\overline\calM_{0,7}/S_5$. 
Boundary divisors $D_T$ map under $P$ to boundary divisors, and we need to compute for a given boundary divisor of $\overline\calM_{0,7}/S_5$ how many $D_T$ appearing in \eqref{eq:keel} map to it, and the degree of the map on each component of the preimage.

We start by computing the coefficient of $D_{12}^2$ in the pushforward of \eqref{eq:keel} under $P$. It only appears in the pushforward of the left-hand side, for those $T$ with $\# T=2$. This is to say $T=\{1,2\}$, and then the degree of the map $P:D_T\to D_{12}^2$ is equal to $5!$, corresponding to arbitrarily relabeling the $5$ points on the component of the curve containing points $\{3,4,5,6,7\}$. For $D_{12}^3$, it appears in the pushforward of terms with $T=\{1,2,a\}$ with $a\ne 3,4$. There are of course $3$ such terms, and for each one the degree of the map $P|_{D_T}$ is equal to $4!$, for relabeling the points not in $T$. For $D_{12}^4$, this is the pushforward of all $D_T$ with $T=\{1,2,a,b\}$ with $4<a<b\le 7$. 
Thus there are $3$ such $T$, and the degree of $P$ on each of them is $2$ (for permuting $a$ and $b$) times $3!$ for permuting the other $3$ points. For $D_{12}^5$, the only $D_T$ on the left in \eqref{eq:keel} mapping to it is for $T=\{1,2,5,6,7\}$. The degree of $P|_{D_T}$ is then $3!$ for relabeling the points $5,6,7$. However, there is here no extra factor of~$2$ for permuting the points $3$ and $4$, since the component of the stable rational curve containing points $3,4$ is anyway parameterized by $\calM_{0,3}$, i.e.~is a point --- this is precisely where the stackiness appears.

To compute the pushforward of the right-hand-side of \eqref{eq:keel}, note that the preimage of $D_1^2$ is only the divisor $D_{\{1,3\}}$, the degree of $P$ on which is $4!$ for permuting the points $4,5,6,7$. For $D_1^3$, the preimage consists of those $D_T$ where $T=\{1,3,a\}$ for $a=5,6,7$. So there are 3 such divisors, and the degree is $2$ (for permuting $a$ and $3$) times $3!$ for each of them. By the symmetry of points $1$ and $2$, the coefficients of $D_1^5$ and $D_1^4$ are the same as those of $D_1^2$ and $D_1^3$, respectively. Thus finally the pushforward $P_*$ of \eqref{eq:keel} is given by
$$
 5!\, D_{12}^2+3\cdot 4!\, D_{12}^3+3\cdot2\cdot 3!\, D_{12}^4+3!\, D_{12}^5=4!\,D_1^2+3\cdot2\cdot 3!\,D_1^3+3\cdot2\cdot 3!\,D_1^4+4!\,D_1^5
$$
Dividing all the coefficients by $6$, we obtain
\begin{equation}\label{rullacorrect}
20 D_{12}^2+12 D_{12}^3+6 D_{12}^4+ D_{12}^5=4 D_1^2+6D_1^3+6D_1^4+4D_1^5.
\end{equation}
This agrees with Rulla's formula \eqref{rulla} except we have $1$ in front of $D_{12}^5$ instead of Rulla's $2$, due to taking the stackiness into account.

For ease of notation, we will denote the quotient of the Deligne-Mumford compactification by
\begin{equation}\label{equ:deftildeM}
\tMq\coloneqq \overline\calM_{0,7}/S_2\times S_5.
\end{equation}
Similarly, we denote the corresponding quotient of the Hassett space $\calH=\overline\calM_{(1/6+\varepsilon)^2,(1/3+\varepsilon)^5}$ by 
\begin{equation}\label{equ:deftildeH}
\tcH\coloneqq\calH/S_2\times S_5.
\end{equation}
 
It is crucial for us that the Hassett space is closely related to the moduli space of cubic surfaces with a line. Indeed, the diagram \eqref{relations} of covers of $\calM_{\rm sm}$, can be extended over the discriminant locus and to the toroidal compactifications. This follows from \cite[Theorem 1.4]{GKS}, which gives an isomorphism between the Naruki compactification $\overline \calN$ and $\oBGm$, and then taking quotients. In particular, we then have an isomorphism 
\begin{equation}\label{equ:extendedisophi}
\varphi: \tcH \cong \oBGell
\end{equation}
between the Hassett space and the toroidal compactification of the moduli space $\calM_{\rm sm,\ell}$ of smooth cubics with a line. We further refer the reader to the discussion of this extension given in \cite{DvGK} and \cite{doran}. 

Let $\sigma:\overline\calM_{0,7}/S_5\to \tMq$ be the double cover (denoted $\sigma$ as we are thinking of this as a quotient by the involution), and (anticipating the future weights in $\calH$) we call the first two points (labeled 1 and 2 by Rulla) {\em light}, and mark them as $\circ$ on the sketches of our curves below, and the remaining 5 permuted points we call {\em heavy}, and mark them by $\bullet$.

We now compute the pushforward of \eqref{rullacorrect} under $\sigma_*$, similarly to how we dealt with $P_*$ above. We denote by $D_{\circ\circ}^i$ the image of $D_{12}^i$ under $\sigma$, and note that $\sigma^{-1}(D_{\circ\circ}^i)=D_{12}^i$ and that $\sigma|_{D_{12}^i}$ has degree two for permuting the points $1$ and $2$, {\em except for $i=2$}, where the degree is equal to 1, as the 
irreducible component containing the node and points $1$ and $2$ is already parameterized by $\calM_{0,3}$. Similarly, we note that $D_\circ^i\coloneqq\sigma(D_1^i)=\sigma(D_1^{7-i})$, and both of these divisors are mapped generically one-to-one to their image under~$\sigma$.

Then the pushforward of \eqref{rullacorrect} under $\sigma$ gives the relation
\begin{equation}\label{sigmarulla}
20 D_{\circ\circ}^2+24 D_{\circ\circ}^3+12D_{\circ\circ}^4+2D_{\circ\circ}^5=8D_\circ^2+12D_\circ^3\in \Pic_\QQ(\tMq)\,.
\end{equation}
As the $6$ divisors appearing in this formula are images of all generators of $\Pic_\QQ(\overline\calM_{0,7}/S_5)$, it follows that
$$
 \dim_\QQ\Pic_\QQ(\tMq)=\dim_\QQ\Pic_\QQ(\overline\calM_{0,7}/S_5)-2=5\,,
 $$
and that this Picard group is generated by divisors $D_{\circ\circ}^2,D_{\circ\circ}^3,D_{\circ\circ}^4,D_{\circ\circ}^5,D_\circ^2,D_\circ^3$, subject to the one relation \eqref{sigmarulla}. (We note that the above can also be computed using \texttt{admcycles} \cite{admcycles}, provided one carefully tracks the generic automorphisms for the correct stacky coefficients.)

We now describe the divisorial contraction $h:\tMq\to \tcH$. We note that the product of symmetric groups still acts on the Hassett space, as the corresponding weights are equal. Thus, we simply need to determine which boundary divisors of $\tMq$ are contracted by the morphism~$h$, which is the question of stability of the generic point of each boundary divisor, with the chosen weights. Indeed, on $D_{\circ\circ}^i$ the weight of the component containing 1,2 is equal to
$$
 2\cdot \frac16+(i-2)\cdot \frac13+1+i\cdot\varepsilon=\frac{i+2}{3}+i\cdot\varepsilon
$$
(where the extra $+1$ is for the node). Thus for $i=2,3$ this component is unstable, while --- since the sum of the weights of all the points is $2+7\varepsilon$ --- for $i=5$ the other component is unstable. 
In all of these cases the unstable rational component will be contracted. If we contract a component (of course attached to the rest of the curve at one node), and if that component 
contains $j\ge 3$ marked points, the dimension of the stratum in the moduli space decreases, as we are mapping $\calM_{0,j+1}$ to a point. Thus $h$ contracts the divisor $D_{\circ\circ}^3$ to a codimension 
two locus. However, the divisors $D_{\circ\circ}^2$ and $D_{\circ\circ}^5$ are mapped under $h$ to divisors in $\tcH$, where we simply think that a generic point represents now an irreducible rational point 
where two light points (resp.~two heavy points) have collided. Furthermore, the generic point of $D_{\circ\circ}^4$ is stable for our choice of the weights, and thus there is no contraction there, either. 
For brevity we denote $\delta_i\coloneqq h(D_{\circ\circ}^i)$.

For a generic point of $D_\circ^2$, the total weight of the points on the component containing only one heavy~$\bullet$ point (and one light~$\circ$ point) is $\tfrac16+\tfrac13+1+2\varepsilon<2$, so this component of the rational nodal curve gets contracted, but again this means we replace $\calM_{0,3}$ by a point, and simply think of the corresponding point in the Hassett space as an irreducible rational curve where a light and a heavy point have collided. Finally, the total weight of the component containing two heavy $\bullet$ points and one light $\circ$ point, at a generic point of $D_\circ^3$ is equal to $\tfrac56+1+3\varepsilon<2$, so $h$ contracts the divisor $D_\circ^3$. (The interested reader will observe 
how our description parallels conditions (i)-(iii) described in \cite[p.~113]{DvGK}). We denote $\gamma\coloneqq h(D_\circ^2)$. For further use, we give the schematic pictures of the generic points of all boundary 
divisors in $\tMq$ and their images in $\tcH$.

\begin{figure}[ht]
\centering\input{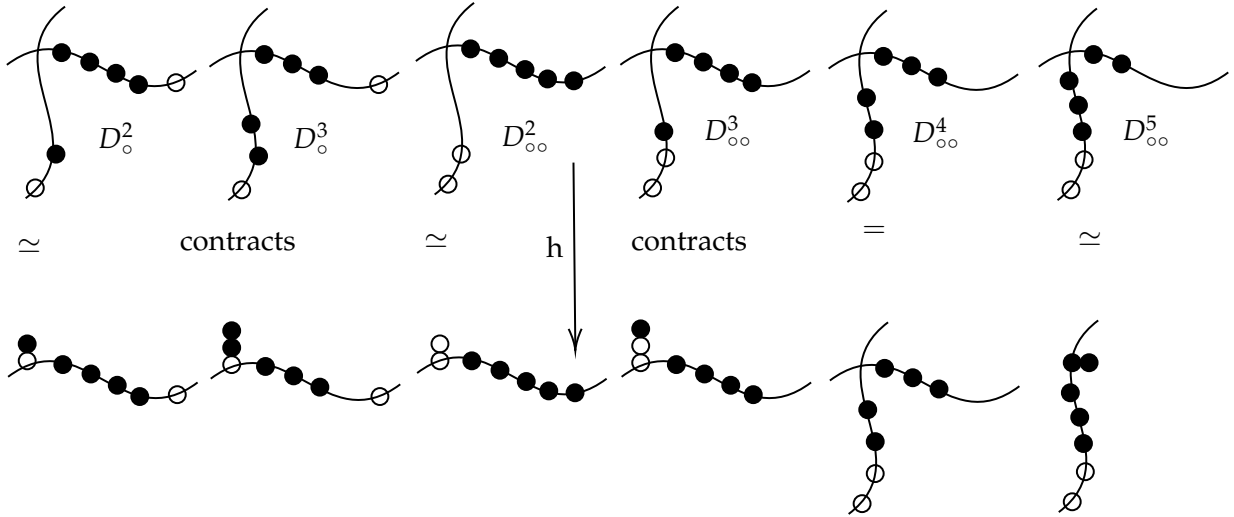}
\caption{The images of boundary divisors under the map $h$}\label{fig:divisors}
\end{figure}

We can summarize the above discussion as follows.
\begin{teo}\label{teo:PicardtildeH}
The Picard group of $\tcH$ is given by
\begin{equation}\label{equ:generatorsPicH}
 \Pic_\QQ(\tcH)=\QQ \delta_2\oplus\QQ\delta_4\oplus\QQ\delta_5\oplus\QQ\gamma/\sim
\end{equation}
where the unique relation~$\sim$ is
\begin{equation}\label{equ:relationsPicH}
 20 \delta_2+12\delta_4+2\delta_5=8\gamma\in \Pic_\QQ(\tcH).
\end{equation}
\end{teo}
\begin{proof}
The relation is obtained by pushing \eqref{sigmarulla} under $h$, resulting in \eqref{equ:relationsPicH}. Here we use that all other classes of divisors are contracted under~$h$. The rest follows from the 
discussion preceding the theorem.
\end{proof}
Here we note that crucially there is only one term on the right-hand-side, i.e. one boundary divisor on $\tcH$ can be expressed as a {\em positive} linear combination of other divisorial boundary strata.

It is now not difficult to determine the effective cone of $\tcH$.
\begin{teo}\label{teo:effconetildeH}
The cone of effective divisors of $\tcH$ is
\begin{equation}\label{equ:effconetildeH}
 \Eff(\tcH)=\RR_{\ge 0}\delta_2+\RR_{\ge 0}\delta_4+\RR_{\ge 0}\delta_5\,.
\end{equation}
\end{teo}
\begin{proof}
Let $E\in \Eff(\tcH)$ and denote its preimage in $\overline\calM_{0,7}/S_5$ by $\tilde E$. By \cite{rulla}, the class $\tilde E$, being that of an effective divisor, must then be equal to some linear combination of $D_\circ^i$ and $D_{\circ\circ}^i$ with non-negative coefficients. Pushing such an expression for $\tilde E$ under the map $h\circ \sigma$ 
expresses the class of $E$ as a non-negative linear combination of the classes $h_*\circ \sigma_*(D_\circ^i)$ and $h_*\circ \sigma_*(D_{\circ\circ}^i)$, which, by the above discussion is 
thus a linear combination of $\delta_2,\delta_4,\delta_5,\gamma$ with non-negative coefficients. Using \eqref{equ:relationsPicH} we can eliminate $\gamma$ from such an expression.
\end{proof}

Since $\tcH \cong \oBGell$, we have thus determined the cone of effective divisors of the toroidal compactification of the moduli space of cubic surfaces with a line. In order to make this geometrically meaningful,
we now explore the meaning of the divisors $\delta_2,\delta_4,\delta_5$ in terms of the moduli of cubic surfaces. For this we use the geometric discussions in \cite[Sec.~3]{DvGK} and \cite{doran}.
\begin{pro}\label{pro:comparingclasses}
Under the isomorphism $\varphi: \tcH \cong \oBGell$ from (\ref{equ:extendedisophi}) we have
\begin{equation}\label{equ:identificationclasses}
\varphi(\delta_2)= \TAin, \quad \varphi(\delta_4)= T_{3A_2,\ell},\quad \varphi(\delta_5)= \TAout, \quad \varphi(\gamma)=\TRin.
\end{equation}
\end{pro}
\begin{proof}
In \cite[Sec.~3]{DvGK} the authors investigate the geometry of the collisions of the $5+2$ points and their relations to the nodal cubic surfaces, and we will now deduce all the claims above except for the one about $\delta_4$ (which does not correspond to nodal cubics) from the lists from the arXiv version of \cite{DvGK}\footnote{See I.~Dolgachev, B.~van Geemen, S.~Kond\=o {\em A complex ball uniformization of the moduli space of cubic surfaces via periods of K3 surfaces}, {\tt arXiv:math/03103421v1}, pp 14--16. }, which contains additional helpful information beyond what is covered in the journal version. 

We start with the divisor $\gamma$. This is where one light and one heavy point collide, which corresponds to case (2) in \cite[pp. 14/15]{DvGK}. 
The special situation is that one of the two tangents to $C_i$ through $p_i$ meets the conic in a point $p_k$. This still leads to a smooth cubic. Blowing up in $p_k$ we see that we obtain an Eckardt point, namely 
the point where the exceptional line and the strict transforms of the conic $C_i$ and the tangent at $p_k$ meet. Hence we land in the Eckardt locus $T_{R,\ell}$. The marked line on the cubic surface $S$ is the strict transform of $C_i$ and goes through the Eckardt point. Hence $\varphi$ maps $\gamma$ to $\TRin$. 
One gets the generic point of $\TRin$ in this way.
Indeed, we can assume that the marked line 
corresponds to $C_i$. By inspection of the lines on $S$, or equivalently analysing the conic bundle associated to the pencil of planes containing the marked line $\ell$, we see that the only way in which 
an Eckardt point with $\ell$ an Eckardt line can arise, is, that one of the tangents to $C_i$ through $p_i$ goes through one of the other points $p_j$. Hence we get that $\varphi(\gamma)=\TRin$. 

The divisor $\delta_2$ is where two light points collide. If we denote the binary forms which give the $5+2$ points by $F_5$ and $F_2$, then this means that $F_2$ has a double root (which is generically disjoint from the roots of $F_5$). This is case (4) in \cite[pp. 14/15]{DvGK}. As is explained there, this means geometrically that $C_i$ degenerates into two lines $C_j=L_1 \cup L_2$ with one line, say $L_1$, containing $3$ of the points $p_j$ and the other line $L_2$ containing the remaining $2$ of the points $p_j$. Then $L_1$ is contracted to give a node on the cubic $S$. The sixth blown up point $p_i$ lies outside $L_1 \cup L_2$ and the tangent to the conic is the line joining $p_i$ and the intersection of $L_1$ and $L_2$. The marked line is again the image of the conic $C_i$ which, in this case, is mapped to the image of the line $L_2$ which goes through the node. Hence $\varphi$ maps $\delta_2$ to $\TAin$ and since this locus is irreducible we obtain that $\varphi(\delta_2)=\TAin$.

We now consider $\delta_5$. Here 2 heavy points, say $p_j,p_k$ collide and this is (5) in \cite{DvGK}. Then the lines $L_j, L_k$ coincide and $p_i,p_j,p_k$ lie on a line which gets contracted to a node. The marked line is the strict transform of the conic $C_i$ and this does not contain the node, thus giving us a point in $\TAout$. Again by irreducibility of $\TAout$ it follows that $\varphi(\delta_5)=\TAout$.

Finally, we turn to the divisor~$\delta_4$, which does not correspond to nodal cubics (i.e.~those with only~$A_1$ singularities). For this we observe that the isomorphism  $\varphi: \tcH \cong \oBGell$ sends the boundary to the boundary. In particular, there must be a boundary divisor in $\tcH$ that maps to the toroidal boundary $T_{3A_2,\ell}$. By the above discussion this can only be the irreducible boundary divisor $\delta_4$ of $\tcH$.
This further implies that $T_{3A_2,\ell}$ is irreducible. The latter can also be seen purely group theoretically and amounts to the statement that the Weyl group $W(D_5)\cong (\ZZ/2\ZZ)^4 \rtimes S_5$ 
acts transitively on the 40 cusps of $\oBGm$. This follows immediately from comparing the orders of possible stabilizers of a cusp in $W(D_5)$ with the stabilizers of $W(E_6)$,
which, by Naruki \cite[p. 22]{naruki}, are an $S_3$-extension of $S_3^3$ and thus have order $6^4$.     
\end{proof}
\begin{rem}\label{rem:numericalcheck}
There is a non-trivial numerical check of our results. Identifying the classes $\delta_2$ and $\delta_5$ with their corresponding divisors in $\oBGell$ and pushing the relation \eqref{equ:relationsPicH} under $\ell_*$, we obtain
\begin{equation}\label{equ:pushforwardclassesell}
 \ell_*(\delta_2)=\ell_*(\TAin)=6 T_{A_1}, \quad \ell_*(\delta_5)=\ell_*(\TAout)=15 T_{A_1}.
\end{equation}
The factor 6 comes form the fact that $\TAin \to T_{A_1}$ has degree 6 (there are 6 lines through a node) and that $\TAout\to T_{A_1}$ has degree 15 (there are 15 lines in the smooth locus of a generic nodal cubic surface), as discussed in the proof of \Cref{lem:AB}. Since the map $\ell$ is unramified along the boundary and has degree 27, we find that 
\begin{equation}\label{equ:pushforwardclassdelta4}
\ell_*(\delta_4)=\ell_*(T_{3A_2,\ell})= 27 T_{3A_2}\,.
\end{equation}
Finally, since there are 3 lines through an Eckardt point we have \begin{equation}\label{equ:pushforwardclassgamma}
\ell_*(\gamma)=\ell_*(\TRin)= 3T_R\,.
\end{equation}
Using these formulae we obtain
\begin{equation}\label{equ:pushforwardrel1}
\ell_*(20\delta_2+12\delta_4+2\delta_5)=(20\cdot 6+2\cdot 15) T_{A_1}+ 27 \cdot12 T_{3A_2}=150T_{A_1}+324 T_{3A_2}\,.
\end{equation}
Similarly 
\begin{equation}
\ell_*(8\gamma)= 24 T_R\,,
\end{equation}
and hence \eqref{equ:relationsPicH} becomes 
\begin{equation}\label{equ:pushforwardrel2}
150T_{A_1}+324 T_{3A_2}= 24 T_R\,,
\end{equation}
which is exactly \eqref{eq:TR}.
\end{rem}
The cover $\ell:\oBGell\to\oBG$ is {\em not} Galois, so a priori it could be quite difficult to use this cover to determine $\Eff(\oBG)$. In this case we are, however, in a fortunate situation, as we shall explain now.

\begin{proof}[Proof of \Cref{teo:effandnef} determining $\Eff(\oBG)$]
As discussed in the introduction, $T_{3A_2}$ is an extremal ray of the cone $\Eff(\oBG)$, as it is contracted under the map to the Baily-Borel compactification. Thus it suffices to prove that $T_{A_1}$ is an extremal effective divisor. Suppose for contradiction that the effective divisor $T_{A_1}\in\Eff(\oBG)$ were not extremal, so that its class could be written as $T_{A_1}=E_1+E_2$. We know from \Cref{lem:AB} that $\ell^*T_{A_1}=2\TAin + \TAout$. From \Cref{teo:PicardtildeH} we know that $\delta_2,\delta_4,\delta_5$ are a $\QQ$-basis of the Picard group of $\oBGell$, and generate the cone of effective divisors 
(again, since the other potential generator $\gamma$ turns out to be their linear combination with positive coefficients). But since the preimage of $T_{3A_2}$ is the irreducible divisor $\delta_4$, pushing and pulling these classes is linearly independent from $T_{A_1}$, and $\delta_2,\delta_5$. Thus $\ell^*(E_1)$ and $\ell^*(E_2)$ must be non-negative linear combinations of $\delta_2$ and $\delta_5$ only, not including $\delta_4$. Pushing forward under $\ell_*$ shows that the $E_i$ are then positive multiples of the class of $T_{A_1}$, and this completes the proof. 
\end{proof}

\section{The nef cones}\label{sec:nefcones}
In this section we determine the cone of nef divisors $\Nef(\oBGell)$ and use it to also deduce $\Nef(\oBG)$. Similarly to the determination in the previous section, this will use the identification from \cite{GKS} of $\oBGell$ with the Hassett moduli space $\tcH$, which is obtained from $\tMq$ by a divisorial contraction described above. Recall that Keel and McKernan \cite{keelmckernan} showed that the cone of effective {\em curves} $\Eff_1(\overline\calM_{0,7})$ is generated by the 1-dimensional boundary strata, which implies that this also holds for the cone of effective curves on the quotient $\Eff_1(\tMq)$. By definition the cone of nef divisors is dual to the cone of effective curves, which is simply to say that a divisor is nef if and only if its intersection with every effective curve is non-negative.

Using \texttt{admcycles} \cite{admcycles}, Johannes Schmitt has kindly enumerated for us the 1-dimensional boundary strata in $\tMq$. There are 24 of them, of course subject to a host of linear relations. For our purposes we will not need to list these 24 curves (by drawing the corresponding dual graphs), but we note that some of them are contracted under $h:\tMq\to \tcH$: such a 1-dimensional boundary stratum is contracted to a point under $h$ if and only if the generic point is a stable curve that has an irreducible component that is a $\PP^1$ with one node and three marked points, at least one of which is light. By inspection, there are 6 contracted curves among 24, numbered 1,4,6,14,16,23 in the spreadsheets available at \url{http://www.math.stonybrook.edu/~sam/IntersectionNumbers.pdf}
and \url{http://www.math.stonybrook.edu/~sam/SageWorksheet.pdf} and in the table below (and also appended after the end of the \LaTeX\ document and as additional pdf files uploaded to the arXiv). For example, the 1-dimensional boundary stratum number 1 parameterizes a chain of 4 $\PP^1$'s, where one of the outside $\PP^1$'s contains two light points and one heavy point (and so for stability the $\PP^1$ next to it has to have one heavy point, the next one has to have one heavy point, and the last one has two heavy points). Similarly, the 1-dimensional stratum number 4 parameterizes a chain of 4 $\PP^1$'s where the first one contains one light and two heavy points, the second contains one light point, the third contains one heavy point, and the last contains two heavy points.

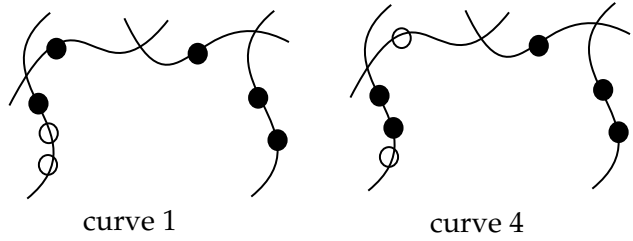
\begin{figure}[ht]
 \centering\tikzset{every picture/.style={line width=0.75pt}} %set default line width to 0.75pt        

\begin{tikzpicture}[x=0.75pt,y=0.75pt,yscale=-1,xscale=1]
%uncomment if require: \path (0,299); %set diagram left start at 0, and has height of 299

%Curve Lines [id:da2000628980764254] 
\draw    (8.33,50) .. controls (45.33,-22) and (51.33,54) .. (88.33,7) ;
%Shape: Ellipse [id:dp04768240820981662] 
\draw   (23.2,80) .. controls (23.2,77.25) and (25.25,75.02) .. (27.78,75.02) .. controls (30.31,75.02) and (32.36,77.25) .. (32.36,80) .. controls (32.36,82.75) and (30.31,84.98) .. (27.78,84.98) .. controls (25.25,84.98) and (23.2,82.75) .. (23.2,80) -- cycle ;
%Shape: Ellipse [id:dp2727524586809118] 
\draw  [fill={rgb, 255:red, 0; green, 0; blue, 0 }  ,fill opacity=1 ] (18.4,49.08) .. controls (18.4,46.33) and (20.45,44.1) .. (22.98,44.1) .. controls (25.51,44.1) and (27.57,46.33) .. (27.57,49.08) .. controls (27.57,51.83) and (25.51,54.06) .. (22.98,54.06) .. controls (20.45,54.06) and (18.4,51.83) .. (18.4,49.08) -- cycle ;
%Shape: Ellipse [id:dp09468533734377793] 
\draw  [fill={rgb, 255:red, 0; green, 0; blue, 0 }  ,fill opacity=1 ] (129.21,46.14) .. controls (129.21,43.39) and (131.26,41.16) .. (133.79,41.16) .. controls (136.32,41.16) and (138.38,43.39) .. (138.38,46.14) .. controls (138.38,48.89) and (136.32,51.12) .. (133.79,51.12) .. controls (131.26,51.12) and (129.21,48.89) .. (129.21,46.14) -- cycle ;
%Shape: Ellipse [id:dp7995370017696992] 
\draw  [fill={rgb, 255:red, 0; green, 0; blue, 0 }  ,fill opacity=1 ] (138.99,67.5) .. controls (138.99,64.75) and (141.04,62.52) .. (143.57,62.52) .. controls (146.1,62.52) and (148.16,64.75) .. (148.16,67.5) .. controls (148.16,70.25) and (146.1,72.48) .. (143.57,72.48) .. controls (141.04,72.48) and (138.99,70.25) .. (138.99,67.5) -- cycle ;
%Shape: Ellipse [id:dp8181823135043309] 
\draw  [fill={rgb, 255:red, 0; green, 0; blue, 0 }  ,fill opacity=1 ] (98.73,23.86) .. controls (98.73,21.11) and (100.78,18.88) .. (103.31,18.88) .. controls (105.85,18.88) and (107.9,21.11) .. (107.9,23.86) .. controls (107.9,26.61) and (105.85,28.84) .. (103.31,28.84) .. controls (100.78,28.84) and (98.73,26.61) .. (98.73,23.86) -- cycle ;
%Shape: Ellipse [id:dp9965448135238286] 
\draw  [fill={rgb, 255:red, 0; green, 0; blue, 0 }  ,fill opacity=1 ] (27.45,21.33) .. controls (27.45,18.58) and (29.51,16.35) .. (32.04,16.35) .. controls (34.57,16.35) and (36.62,18.58) .. (36.62,21.33) .. controls (36.62,24.08) and (34.57,26.31) .. (32.04,26.31) .. controls (29.51,26.31) and (27.45,24.08) .. (27.45,21.33) -- cycle ;
%Shape: Ellipse [id:dp9783605066387302] 
\draw   (23.58,64.03) .. controls (23.58,61.28) and (25.63,59.05) .. (28.16,59.05) .. controls (30.69,59.05) and (32.75,61.28) .. (32.75,64.03) .. controls (32.75,66.78) and (30.69,69.01) .. (28.16,69.01) .. controls (25.63,69.01) and (23.58,66.78) .. (23.58,64.03) -- cycle ;
%Curve Lines [id:da7389274690159675] 
\draw    (66.33,6) .. controls (93.33,64) and (102.33,-7) .. (149.33,18) ;
%Curve Lines [id:da44235930073189844] 
\draw    (130.33,96) .. controls (171.07,62.8) and (101.17,35.41) .. (141.9,2.21) ;
%Curve Lines [id:da36549924145290724] 
\draw    (17.33,97) .. controls (58.07,63.8) and (-11.83,36.41) .. (28.9,3.21) ;
%Curve Lines [id:da12453281549457551] 
\draw    (180.33,46) .. controls (217.33,-26) and (223.33,50) .. (260.33,3) ;
%Shape: Ellipse [id:dp45299226379863433] 
\draw   (195.2,76) .. controls (195.2,73.25) and (197.25,71.02) .. (199.78,71.02) .. controls (202.31,71.02) and (204.36,73.25) .. (204.36,76) .. controls (204.36,78.75) and (202.31,80.98) .. (199.78,80.98) .. controls (197.25,80.98) and (195.2,78.75) .. (195.2,76) -- cycle ;
%Shape: Ellipse [id:dp24954036808044777] 
\draw  [fill={rgb, 255:red, 0; green, 0; blue, 0 }  ,fill opacity=1 ] (190.4,45.08) .. controls (190.4,42.33) and (192.45,40.1) .. (194.98,40.1) .. controls (197.51,40.1) and (199.57,42.33) .. (199.57,45.08) .. controls (199.57,47.83) and (197.51,50.06) .. (194.98,50.06) .. controls (192.45,50.06) and (190.4,47.83) .. (190.4,45.08) -- cycle ;
%Shape: Ellipse [id:dp913615423233245] 
\draw  [fill={rgb, 255:red, 0; green, 0; blue, 0 }  ,fill opacity=1 ] (303.21,42.14) .. controls (303.21,39.39) and (305.26,37.16) .. (307.79,37.16) .. controls (310.32,37.16) and (312.38,39.39) .. (312.38,42.14) .. controls (312.38,44.89) and (310.32,47.12) .. (307.79,47.12) .. controls (305.26,47.12) and (303.21,44.89) .. (303.21,42.14) -- cycle ;
%Shape: Ellipse [id:dp13133394327474845] 
\draw  [fill={rgb, 255:red, 0; green, 0; blue, 0 }  ,fill opacity=1 ] (310.99,63.5) .. controls (310.99,60.75) and (313.04,58.52) .. (315.57,58.52) .. controls (318.1,58.52) and (320.16,60.75) .. (320.16,63.5) .. controls (320.16,66.25) and (318.1,68.48) .. (315.57,68.48) .. controls (313.04,68.48) and (310.99,66.25) .. (310.99,63.5) -- cycle ;
%Shape: Ellipse [id:dp7218884680165876] 
\draw  [fill={rgb, 255:red, 0; green, 0; blue, 0 }  ,fill opacity=1 ] (270.73,19.86) .. controls (270.73,17.11) and (272.78,14.88) .. (275.31,14.88) .. controls (277.85,14.88) and (279.9,17.11) .. (279.9,19.86) .. controls (279.9,22.61) and (277.85,24.84) .. (275.31,24.84) .. controls (272.78,24.84) and (270.73,22.61) .. (270.73,19.86) -- cycle ;
%Shape: Ellipse [id:dp3282112951299849] 
\draw  [fill={rgb, 255:red, 0; green, 0; blue, 0 }  ,fill opacity=1 ] (197.45,61.33) .. controls (197.45,58.58) and (199.51,56.35) .. (202.04,56.35) .. controls (204.57,56.35) and (206.62,58.58) .. (206.62,61.33) .. controls (206.62,64.08) and (204.57,66.31) .. (202.04,66.31) .. controls (199.51,66.31) and (197.45,64.08) .. (197.45,61.33) -- cycle ;
%Shape: Ellipse [id:dp7663066577603328] 
\draw   (201.58,16.03) .. controls (201.58,13.28) and (203.63,11.05) .. (206.16,11.05) .. controls (208.69,11.05) and (210.75,13.28) .. (210.75,16.03) .. controls (210.75,18.78) and (208.69,21.01) .. (206.16,21.01) .. controls (203.63,21.01) and (201.58,18.78) .. (201.58,16.03) -- cycle ;
%Curve Lines [id:da4753800526300609] 
\draw    (238.33,2) .. controls (265.33,60) and (274.33,-11) .. (321.33,14) ;
%Curve Lines [id:da8728470070865737] 
\draw    (302.33,92) .. controls (343.07,58.8) and (273.17,31.41) .. (313.9,-1.79) ;
%Curve Lines [id:da19917959917321615] 
\draw    (189.33,93) .. controls (230.07,59.8) and (160.17,32.41) .. (200.9,-0.79) ;

% Text Node
\draw (44,101) node [anchor=north west][inner sep=0.75pt]   [align=left] {curve 1};
% Text Node
\draw (219,102) node [anchor=north west][inner sep=0.75pt]   [align=left] {curve 4};

\end{tikzpicture}
 \caption{Two 1-dimensional boundary strata in $\tMq$ contracted to points under $h$}\label{fig:contractedcurves}
\end{figure}

Furthermore, Johannes Schmitt used \texttt{admcycles} to compute for us the intersection numbers of these curves with the 6 boundary divisors. 
The \texttt{admcycles} package in fact uses Pixton's original code for intersection computations, and since proper symmetrization has not yet been implemented, it means that each of the boundary divisors used for these intersection computations is in fact some multiple of one of $D_{\circ\circ}^i$ or $D_\circ^i$. However, since the $6$ divisors satisfy a unique linear relation, all of these scaling factors can be determined uniquely up to an overall factor. Rescaling the \texttt{admcycles} numbers (by dividing them by $120,48,24,12,24,12$ respectively for the 6 divisors below, in the order listed), the resulting intersection matrix is as follows:
\begin{equation}\label{intmatrix}
\scalebox{0.75}{$\begin{array}{l|rrrrrrrrrrrrrrrrrrrrrrrr}
\hline
&1&2&3&4&5&6&7&8&9&10&11&12&13&14&15&16&17&18&19&20&21&22&23&24\\
\hline
D_{\circ\circ}^2&2&2&-2&0&0&2&-2&0&0&0&0&0&0&0&0&0&0&0&0&0&0&0&0&0\\
D_{\circ\circ}^3&-1&0&2&0&1&-1&1&0&1&-1&-1&0&0&0&0&0&0&0&0&0&0&0&0&0\\
D_{\circ\circ}^4&0&-1&-1&0&-1&0&2&1&0&2&2&-1&-1&0&0&0&1&1&-1&-1&-1&-1&0&0\\
D_{\circ\circ}^5&0&-2&2&2&0&0&-4&-2&-4&0&0&6&6&2&2&2&-2&-2&6&6&6&6&2&2\\
D_{\circ}^2&2&0&0&2&0&2&0&1&-1&0&0&0&0&2&-1&2&1&-2&0&0&0&0&2&-1\\
D_{\circ}^3&0&2&0&-1&1&0&0&0&2&0&0&0&0&-1&1&-1&0&2&0&0&0&0&-1&1\\
\hline
\end{array}$}
\end{equation}
This intersection matrix determines the nef cone $\Nef(\tMq)$, but we will not need this full information, as we will use the following (probably well known) lemma about nef cones to determine the nef cone of the Hassett space.
For the reader's sake we recall

\begin{lem}\label{lem:pullnef}
Let $X$ and $Y$ be irreducible $\QQ$-factorial varieties and let $f: X\to Y$ be a morphism which is either birational or dominant and finite. Then the following holds:
$$
f^*(\Nef(Y)) = \Nef(X) \cap f^*(\NS(Y).
$$
\end{lem}
\begin{proof}
It follows immediately from the projection formula that the pullback of a nef divisor is again nef, which shows the inclusion $f^*(\Nef(Y)) \subset \Nef(X) \cap f^*(\NS(Y))$. To see the converse, we fist assume $f$ to be a birational morphism. Let $D$ be a divisor class on $Y$ such that $f^*(D)$ is nef on $X$. We want to show
that $D$ is nef on $Y$. Let $C$ be an irreducible curve on $Y$. We want to prove that $D.C \geq 0$. This follows from the projection formula
$$
(f^*D).\widetilde C=D. f_*(\widetilde C)
$$
provided one finds a curve $\widetilde C \subset X$ such that $f_*(\widetilde C) =mC$ for some positive multiple $m$. If $C$ is not contained in the exceptional locus, then we can simply take
the strict transform $\widetilde C$ of $C$ in $X$.

If $C$ is contained in the exceptional locus, then choose an irreducible component~$E$ of the preimage $f^{-1}(C)$ that dominates~$C$. By embedding $E$ in some projective space and by taking repeated hyperplane sections, one can construct an algebraic subset $H$ of $E$ which intersects a generic fiber of the restriction $f\mid_E: E \to C$ in a finite number of points. In particular, we can find a component of $H$ which is an irreducible curve $\widetilde C$ such that $f\mid_{\widetilde C}: \widetilde C \to C$ is finite.

The proof for a finite surjective morphism is even easier as one can simply take the full preimage of $C$.
\end{proof}
We will of course want to apply this lemma to $h$ to determine the nef cone of the Hassett space, but doing so requires first determining the pullbacks of divisors under $h^*$.
\begin{lem}\label{lem:divpullbacks}
The pullbacks of the boundary divisors under the map $h$ are as follows:
$$
 h^*\delta_2=D_{\circ\circ}^2+2D_{\circ\circ}^3;\qquad h^*\delta_4=D_{\circ\circ}^4;
 \qquad h^*\delta_5=D_{\circ\circ}^5+2D_\circ^3;\qquad h^*\gamma=D_\circ^2+2D_\circ^3+2D_{\circ\circ}^3\,.
$$
\end{lem}
We note that of course \eqref{equ:relationsPicH} expresses $\gamma$ in terms of the other divisors, and thus allows us to compute its pullback, but the proof below gives an extra numerical cross-check.
\begin{proof}
In the previous section we determined the images of all boundary divisors under $h$, in particular observing that $h$ contracts $D_{\circ\circ}^3$ and $D_\circ^3$ to codimension two loci, as pictured in \Cref{fig:divisors}. We now check which boundary divisors of $\tcH$ contain these codimension two strata. The image $h(D_\circ^3)$ parameterizes $\PP^1$ with 3 marked heavy points, one marked light point, and one point where one light and two heavy points collided. Such a $\PP^1$ with weighted points is a degeneration of either the $\PP^1$ where a light and a heavy point collided (and the other heavy point is separate) --- which is a generic point of $h(D_\circ^2)$ --- or of a $\PP^1$ where two heavy points collided, but a light point is separate --- which is a generic point of $h(D_{\circ\circ}^5)$. Thus $h(D_\circ^2)\subset h(D_\circ^3),h(D_{\circ\circ}^5)$. Similarly $h(D_{\circ\circ}^3)\subset h(D_{\circ\circ}^2),h(D_\circ^2)$. The pullback of each boundary divisor under $h^*$ is a linear combination of the irreducible components of its preimage, and thus we know which divisors on $\tMq$ enter in each pullback. Since the images of those boundary divisors that are not contracted by~$h$ are all distinct, we know that they appear in the pullbacks of their images with coefficient one, and it remains to determine the coefficients with which the contracted divisors appear.

While of course it is possible to determine the multiplicities by local computations, there is an easier way. Indeed, recall that the map $h$ contracts six 1-dimensional boundary strata. By inspection of \eqref{intmatrix}, we see that the contracted curves number 1 and 6, and respectively the curves number 4,14,16,23 are numerically equivalent. Each curve contracted under~$h$ must have zero intersection number with each divisor pulled back from $\tcH$. Denoting the classes of these contracted curves $C_1$ and $C_4$, suppose that $h^*\delta_2=D_{\circ\circ}^2+\alpha D_{\circ\circ}^3$. Then $0=h^*\delta_2.C_1=2-\alpha$, which gives $\alpha=2$. Similarly $0=h^*\gamma.C_1$ shows that the coefficient of $D_{\circ\circ}^3$ in this pullback is also 2. Then $0=h^*\delta_5.C_4$ determines the pullback of $\delta_5$, and finally $0=h^*\gamma.C_4$ gives the coefficient of $D_\circ^3$ in $h^*\gamma$. 
\end{proof}

This allows us to complete the determination of the nef cone of $\oBGell$.
\begin{pro}\label{pro:NefMl}
The nef cone $\Nef(\tcH)$ consists of divisors of the form $a\delta_2+b\delta_4+c\delta_5$, where $a,b,c,$ satisfy the inequalities
$$
 b\ge a;\quad b/2\ge c\ge b/6; \quad 2a+2c\ge b
$$
(we note that $b\ge 2c$, $2a+2c\ge b$ together imply that $a\ge 0$, and thus $b\ge 0$ and thus $c\ge 0$).

Equivalently, this nef cone is generated by the divisors
\begin{equation}\label{equ:generatorsnefconeH}
 2\delta_4+\delta_5; \qquad 2\delta_2+2\delta_4+\delta_5; \qquad 2\delta_2+6\delta_4+\delta_5; \qquad 6\delta_2+6\delta_4+\delta_5.
\end{equation}
\end{pro}
\begin{proof}
Applying the lemma above to our situation of $h:\tMq\to\tcH$, this simply says that a divisor $N$ on $\tcH$ is nef if and only if its pullback $h^*N$ is nef. By using \Cref{lem:divpullbacks}, we compute the intersection numbers of the pullback of the basis of the space of divisors on $\tcH$ with all 1-dimensional boundary strata on $\tMq$:
\begin{equation}\label{int_onH}
\scalebox{0.75}{$\begin{array}{l|rrrrrrrrrrrrrrrrrrrrrrrr}
\hline
&1&2&3&4&5&6&7&8&9&10&11&12&13&14&15&16&17&18&19&20&21&22&23&24\\
\hline
h^*\delta_2&\textcolor{red}0&2&\textcolor{blue}2&\textcolor{red}0&\textcolor{blue}2&\textcolor{red}0
&\textcolor{blue}0&0&2&-2&\textcolor{blue}{-2}&0
&\textcolor{blue}0&\textcolor{red}0&0&\textcolor{red}0&\textcolor{blue}0&0
&\textcolor{blue}0&\textcolor{blue}0&\textcolor{blue}0&\textcolor{blue}0&\textcolor{red}0&\textcolor{blue}0\\
h^*\delta_4&\textcolor{red}0&-1&\textcolor{blue}{-1}&\textcolor{red}0&\textcolor{blue}{-1}&\textcolor{red}0
&\textcolor{blue}2&1&0&2&\textcolor{blue}2&-1
&\textcolor{blue}{-1}&\textcolor{red}0&0&\textcolor{red}0&\textcolor{blue}1&1
&\textcolor{blue}{-1}&\textcolor{blue}{-1}&\textcolor{blue}{-1}&\textcolor{blue}{-1}&\textcolor{red}0&\textcolor{blue}0\\
h^*\delta_5&\textcolor{red}0&2&\textcolor{blue}2&\textcolor{red}0&\textcolor{blue}2&\textcolor{red}0
&\textcolor{blue}{-4}&-2&0&0&\textcolor{blue}0&6
&\textcolor{blue}6&\textcolor{red}0&4&\textcolor{red}0&\textcolor{blue}{-2}&2
&\textcolor{blue}6&\textcolor{blue}6&\textcolor{blue}6&\textcolor{blue}6&\textcolor{red}0&\textcolor{blue}4\\
\hline
=column&&&2&&2&&\approx8&&&&10&&12&&&&8&&12&12&12&12&&15
\end{array}$}
\end{equation}
Here we have labeled \textcolor{red}{red} the zero columns, which correspond to those curves that are contracted under $h$, and have labeled \textcolor{blue}{blue} those columns that are proportional to another column --- and we have written below which column that is.

Thus the remaining black columns are the ones that impose distinct conditions for the pullback of a linear combination $a\delta_2+b\delta_4+c\delta_5$ to have a non-negative intersection with the corresponding curve. After dividing by the gcd of the entries in each column, intersecting with these generators of the cone of effective curves gives the following set of inequalities defining $\Nef(\tcH)$:
\begin{equation}
\begin{aligned}
\hbox{column } 2:\quad&2a-b+2c\ge 0\\
\hbox{column } 8:\quad&b-2c\ge 0\\
\hbox{column } 9:\quad&a\ge 0\\
\hbox{column }10:\quad&b-a\ge 0\\
\hbox{column }12:\quad&6c-b\ge 0\\
\hbox{column }15:\quad&c\ge 0\\
\hbox{column } 18:\quad&b+2c\ge 0
\end{aligned}
\end{equation}
As noted in the statement of the proposition, the fact that $a,b,c\ge 0$ (and thus also $b+2c\ge 0$) is implied by the other inequalities, and thus the 4 inequalities defining the cone of nef divisors are as stated. To find the extremal rays of the cone, 2 (or, in principle, more) out of 4 of these inequalities must become equalities, while the other inequalities must be satisfied. Inspecting all 6 possibilities gives the generators: for example if $b=a$ and $b=2c$, then we get the generator $(2,2,1)$, while if $b=a$ and $2a+2c=b$, this would imply $c=-b/2$, and depending on signs, one of the inequalities $6c\ge b\ge 2c$ would fail.
\end{proof}
We can now use this to finish the determination of $\Nef(\oBGell)$ and $\Nef(\oBG)$.
\begin{proof}[Proof of \Cref{teo:effnefline} and \Cref{teo:effandnef} for nef divisors]
\Cref{pro:NefMl} together with the geometric interpretation of the boundary divisors~$\delta$ given in \Cref{pro:comparingclasses} immediately imply the result for $\Nef(\oBGell)$.

We recall from \Cref{lem:AB} and from \Cref{pro:comparingclasses} (where we also use that the quotients are unramified over the boundary) that for the non-Galois cover $\ell:\oBGell\to \oBG$ the pullbacks of the boundary divisors are, viewed from the point of view of~$\tcH$,
$$
\ell^*T_{A_1}=2\delta_2+\delta_5; \quad \ell^*T_{3A_2}=\delta_4\,.
$$
By \Cref{lem:pullnef} a divisor class $\alpha T_{A_1}+\beta T_{3A_2}$ is nef if and only if its pullback under $\ell^*$ is nef. Rescaling the (necessarily positive) coefficient $\alpha$ to 1, we need to determine under which conditions on $\beta$ the class
$$
 \ell^*(T_{A_1}+\beta T_{3A_2})=2\delta_2+\beta \delta_4+\delta_5
$$
lies in $\Nef(\oBGell)$. By luck, we precisely hit the two generating extremal rays $(2,2,1)$ and $(2,6,1)$ of $\Nef(\oBGell)$ found in \Cref{pro:NefMl}, for the values of $\beta$ equal to $2$ and $6$, respectively.
\end{proof}
\begin{rem}
As a major numerical check for all of our computations above, note that the pullback $p^*\calO(1)$ of the polarization from $\BBG$ is definitely a nef, but not ample, line bundle on $\oBG$, as it is a pullback of a nef bundle, and restricts to a trivial bundle on $T_{3A_2}$. As computed above
$$
 4p^*\calO(1)=\ell^*(T_{A_1}+6T_{3A_2})=2\delta_2+6\delta_4+\delta_5
$$
is indeed one extremal ray of the nef cone $\Nef(\oBGell)$.
\end{rem}

\section{Intersection theory on the Kirwan blowup}\label{sec:intersectionmKirwan}
In this section we determine the intersection theory of divisors on~$\MK$, proving \Cref{teo:intersectiontheory}. For this we choose $D_{3A_2}$ and $L=\pi^*(\calO_{\PP(1,2,3,4,5)}(1))$ as the generators of $\Pic(\MK)=\NS(\MK)$. As discussed after the statement of \Cref{teo:intersectiontheory}, 
it suffices to determine the top intersection numbers $L^4$, which we already know to be $1/5!6^4$, and $D_{3A_2}^4$.

\begin{pro} The top self intersection of the exceptional divisor $D_{3A_2}$ on $\MK$ is given by:
$$
D_{3A_2}^4=-\frac{1}{9 \cdot 56}.
$$
\end{pro}
\begin{proof}
We start with a brief recollection of the construction of $\MK$ and $D_{3A_2}$. In particular, setting $Z$ to be the orbit of the $3A_2$ cubic in the Hilbert scheme of cubic surfaces $H=\PP^{19}$, we have a commutative diagram
\begin{equation}\label{E:ConstD3A2}
\xymatrix{
E^{ss}_{}\ar@{^(->}[r]\ar[d]^{/\operatorname{SL}_4}&(\operatorname{Bl}_Z(\PP^{19})^{ss})^{ss}\ar[r]^<>(0.5){\tilde \pi} \ar[d]_{\tilde p}^{/\operatorname{SL}_4}&(\PP^{19})^{ss}\ar[d]_p^{/\operatorname{SL}_4}\\
D_{3A_2}\ar@{^(->}[r]& \MK\ar[r]^\pi& \GIT
}
\end{equation}
where we have blown up the semistable locus $(\PP^{19})^{ss}$ along the orbit~$Z$, and then taken the semistable locus in the blowup with respect to the polarization given by Kirwan. In our case the Kirwan desingularization is obtained by a single blowup. Indeed, there is only one closed orbit with positive dimensional stabilizer to begin with, and one can show by a direct computation 
that after the first blow-up there are no closed 
orbits with positive dimensional stabilizers (see \cite[Sect. 3, \S5]{kirwanhyp} and \cite[Sect. 6]{ZhangCubic} for more details, as well as \cite[\S 2.3.2]{cubics} which reviews how this follows more generally from the fact that there is only one closed orbit to begin with). In particular, after the first blowup every semistable point  is in fact stable (i.e., every point has finite stabilizers).
We descend $\calO_{(\operatorname{Bl}_Z(\PP^{19})^{ss})^{ss}}(E^{ss})$ to a $\QQ$-line bundle $\calO_{\MK}(D_{3A_2})$, and in turn, further descend $\calO_{(\operatorname{Bl}_Z(\PP^{19})^{ss})^{ss}}(E^{ss})|_{E^{ss}}$ to 
a $\QQ$-line bundle $\calO_{\MK}(D_{3A_2})|_{D_{3A2}}$.

We also have another description of $D_{3A_2}$, as $D_{3A_2}\cong \PP^5/\!\!/_{\calO_{\PP^5}(3)} G_{3A_2}$, where $\PP^5$ is the exceptional divisor in the blowup at the origin of an explicit $\CC^6$ Luna Slice for the $3A_2$ cubic surface, and $G_{3A_2}$ is the stabilizer of the $3A_2$ cubic surface acting on this projectivization via an explicit linearization of the action of $\calO_{\PP^5}(3)$ (all of this is induced by the action on the Luna Slice). This Luna Slice construction is compatible with the global construction in the sense that one can view the blowup of the Luna Slice as sitting in the blowup of the ambient $\PP^{19}$, and after restriction, $\calO_{(\operatorname{Bl}_Z(\PP^{19})^{ss})^{ss}}(E^{ss})|_{\PP^5}\cong \calO_{\operatorname{Bl}_0\CC^6}(\PP^5)|_{\PP^5}\cong \calO_{\PP^5}(-1)$. 
In other words, under the isomorphism $D_{3A_2}\cong \PP^5/\!\!/_{\calO_{\PP^5}(3)} G_{3A_2}$, we have that $\calO_{\PP^5}(-1)$ descends to $D_{3A_2}$ to give $\calO_{\MK}(D_{3A_2})|_{D_{3A2}}$.

We now recall the quotient 
$D_{3A_2}\cong \PP^5/\!\!/_{\calO_{\PP^5}(3)} G_{3A_2}$ in a little more detail. Indeed, $G_{3A_2}$ is an extension of $S_3$ by $(\CC^*)^2$, so that we may first consider $X_{P_a}:= \PP^5/\!\!/_{\calO_{\PP^5}(3)} (\CC^*)^2$ and then take the finite $S_3$ quotient, $D_{3A_2}\cong X_{P_a}/S_3$. We have previously described $ X_{P_a}$ as an explicit toric variety given by a rational polytope $P_a$ (see \cite[Lem.~8.2]{Nonisom}), whose volume is $\operatorname{vol}P_a=\frac{3}{56}$ \footnote{We thank Mathieu Dutour Sikiri\'c for computing this for us.}. We in fact translate our projective GIT problem into an affine GIT problem, and the polytope $P_a$ and the corresponding lattice $M$ giving the toric data we obtain from the process in \cite[Ch.~14]{CLStoric}. Indeed, \cite[Thm.~14.2.13]{CLStoric} gives, in our language an identification of the invariant ring as $R(\PP^5,\calO_{\PP^5}(3))^{G_{3A_2}} = \CC[C(P_a)\cap (M\oplus \ZZ)]$ (this notation on the right means the graded ring obtained by placing $P_a$ at height $1$, so to speak, and then taking the cone through the origin). In other words, $\calO_{\PP^5}(3)$ descends to the $\QQ$-line bundle $L_a$ on $X_{P_a}$ given by the ``$\calO(1)$'' on the toric variety dictated by the polytope. (To reduce to the standard situation, by taking Veronese subrings, one may assume the polytope is a lattice polytope; keeping track of the scaling and the Veronese subring, one obtains the stated result.) In summary, recalling that $\calO_{\PP^5}(-1)$ descends to give~$\calO_{\MK}(D_{3A_2})|_{D_{3A2}}$ on~$D_{3A_2}$, we note that $\calO_{\PP^5}(-1)$ descends to $X_{P_a}$ to give the $\QQ$-line bundle $-\frac{1}{3}L_a$. Observing that $L_a^3$ is given by the normalized volume of the polytope $P_a$ (again, to see this, take a Veronese subring to get an integral polytope, and then keep track of the scaling and the Veronese subring), we have
$$
L_a^3= 3! \operatorname{vol}P_a= 6\cdot \frac{3}{56}.
$$

Since the map $X_{P_a}\to X_{P_a}/S_3 = D_{3A_2}$ has degree~$6$, we finally obtain
$$
D_{3A_2}^4= (D_{3A_2}|_{D_{3A_2}})^3= \frac{1}{6}\cdot (-\frac{1}{3}L_a)^3= \frac{1}{6}\cdot \frac{-1}{27}\cdot 6\cdot \frac{3}{56}=-\frac{1}{9}\cdot\frac{1}{56}.
$$
\end{proof}

We next use the commutative diagram \eqref{E:ConstD3A2} to compute the pullbacks
\begin{align*}
\pi^*D&=D_{A_1}+6D_{3A_2}\\
\pi^*R & = D_R+30D_{3A_2}.
\end{align*}
Indeed, we have $p^*D= \Delta$, where $\Delta\subseteq H^{ss}=(\PP^{19})^{ss}$ is the discriminant in the Hilbert scheme of cubic surfaces. Then locally, $\Delta$ looks like the product of three cusps, times some smooth factors, so has multiplicity $6$. 
So we have $\tilde \pi^*\Delta = \widetilde \Delta + 6 E$, where $E$ is the exceptional divisor. On the other hand $\tilde p^*(D_{A_1}+\alpha D_{3A_2})= \widetilde \Delta + \alpha E$. We conclude that $\alpha=6$.

The computation for the strict transform of the Eckardt divisor is similar, except that here we have ramification, i.e., $p^*R= 2R_H$, where $R_H$ is the Eckardt divisor on $H^{ss}$. Locally, we saw in \cite[Lem.~3.5]{Nonisom} that $R_H$ had multiplicity $15$ along the $3A_2$ locus. So we have that $\tilde \pi^* p^*R = 2\widetilde R_H+ 30 E$. On the other hand, $\tilde p^*(D_R+\alpha D_{3A_2})= 2\widetilde R_H + \alpha E$. One concludes that $\alpha = 30$.

Given these computations, and the fact that we have expressed $D$ and $R$ in terms of the Hodge class $\lambda$ (and abusing notation by writing $\lambda = \pi^*\lambda$), we have
$$
 D_{A_1}=24 \lambda - 6D_{3A_2} \ \ \ \text{ and } \ \ \ D_R= 150 \lambda -30D_{3A_2}.
$$
For the slopes this means that
$$
s(D_{A_1})=4<5 =s(D_R),
$$
compared to 
$$
s(T_{A_1})=4< 6.25 = s(T_{R})
$$
for the toroidal compactification.

\begin{proof}[Proof of \Cref{pro:intersect}]
Recall from \cite[Prop.~5.8, Cor.~6.8]{Nonisom} that we have
\begin{align*}
K_{\MK} &= \pi^* K_{ \GIT}+20D_{3A_2}=-\frac{15}{4}\pi^*D+20D_{3A_2}\,,\\
K_{\oBG} &= p^* K_{\GIT}+16T_{3A_2}=-\frac{15}{4}p^*D+16T_{3A_2}\,.
\end{align*}
It follows that 
$$
K_{\oBG}^4=(p^*K_{\BBG}+16T_{3A_2})^4=K_{\BBG}^4+16^4T_{3A_2}^4=\tfrac{3375}{8}-\tfrac{16^4}{6^3}=\tfrac{3375}8-\tfrac{8192}{27} 
$$
$$
K_{\MK}^4=(p^*K_{\GIT}+20D_{3A_2})^4=K_{\GIT}^4+20^4D_{3A_2}^4=\tfrac{3375}{8}-\tfrac{20^4}{9\cdot 56}=\tfrac{3375}8-\tfrac{20000}{63}.
$$
\end{proof}
We note that this fully agrees with our argument in \cite[Section 7]{Nonisom} that the two varieties are not $K$-equivalent. There we argued that 
that the prime factor $5^4$ from the coefficient $20^4$ cannot be canceled by the denominator. For this we described the possible orders of automorphism groups, which are all products of prime powers of 2 and 3, 
except for $\ZZ_{21}$ which fits with the denominators $8$ and $63$.

%%%%%%%%%%%%%%%%%%%%%% BIBLIOGRAPHY %%%%%%%%%%%%%%%%%%%%%%%%%%%%%%%%%%%%%
\bibliographystyle{amsalpha}
\bibliography{cubicbirat}

\end{document}